\documentclass{2018-Shai}
\author{Ziyou Wu, Steven L. Brunton, \& Shai Revzen}
\newcommand{\ShortTitle}{Challenges in Dynamic Mode Decomposition}
\journalname{\ShortTitle}
\title{\ShortTitle}

\usepackage{graphicx,setspace}
\graphicspath{{figures/}{./}}
\usepackage[table]{xcolor}
\usepackage{amssymb,amsmath,amsthm}
\usepackage{dsfont}
\usepackage[mathscr]{eucal}
\usepackage{amsthm}
\usepackage[compress, numbers]{natbib}
\usepackage{listings}

\usepackage{marvosym}
\usepackage[
    hypertexnames,%
    citecolor=blue,%
    colorlinks=true,%
    linkcolor=red%
]{hyperref}
\usepackage[all]{hypcap}
\usepackage{placeins}

\newcommand{\concept}[1]{\textit{#1}}
\newcommand{\R}{\mathbb{R}}
\newcommand{\Cont}{\mathcal{C}}

\newcommand{\C}{\mathbb{C}}

\newcommand{\N}{\mathbb{N}}

\newcommand{\opK}{\mathcal{K}}
\newcommand{\T}{\mathsf{T}}
\newcommand{\M}{\mathbf{M}}
\newcommand{\sX}{\mathrm{\mathbf{X}}}

\newcommand{\Deriv}{\mathrm{\mathbf{D}}}
\newcommand{\Aut}{\text{Aut}}

\newcommand{\e}{\mathrm{e}}
\newcommand{\XMat}{\tilde{\Lambda}}
\newcommand{\Xeig}[1]{\lambda_{[{#1}]}}

\newcommand{\refFig}[1]{figure~\ref{fig:#1}}

\newcommand{\refSec}[1]{\S\ref{sec:#1}}

\newcommand{\refEqn}[1]{eqn.~\ref{eqn:#1}}

\newcommand{\packList}{\setlength{\parsep}{0pt}\setlength{\parskip}{0pt}\setlength{\itemsep}{0pt}}

\newtheorem{theorem}{Theorem}[section]
\newtheorem{lemma}[theorem]{Lemma}

\begin{document}
\maketitle
\tableofcontents

\section{Abstract}
Dynamic Mode Decomposition (DMD) is a powerful tool for extracting spatial and temporal patterns from multi-dimensional time series, and it has been used successfully in a wide range of fields, including fluid mechanics, robotics, and neuroscience.   
Two of the main challenges remaining in DMD research are noise sensitivity and issues related to Krylov space closure when modeling nonlinear systems. 
Here, we investigate the combination of noise and nonlinearity in a controlled setting, by studying a class of systems with linear latent dynamics which are observed via multinomial observables. 
Our numerical models include system and measurement noise.
We explore the influences of dataset metrics, the spectrum of the latent dynamics, the normality of the system matrix, and the geometry of the dynamics. 
Our results show that even for these very mildly nonlinear conditions, DMD methods often fail to recover the spectrum and can have poor predictive ability. 
Our work is motivated by our experience modeling multilegged robot data, where we have encountered great difficulty in reconstructing time series for oscillatory systems with slow transients, which decay only slightly faster than a period. 

\section{Introduction}
The dynamic mode decomposition (DMD) is a dimensionality reduction and modeling approach that was developed in the fluids community~\cite{Schmid2010jfm,schmid2011applications}.  
DMD is a modal decomposition technique~\cite{Taira2017aiaa,Taira2020aiaa} that decomposes a matrix of high-dimensional time series data into a set of spatial coherent structures that exhibit the same linear dynamics in time (e.g., oscillations, exponential growth/decay). 
In this way, the DMD may be thought of as an algorithmic descendent of the principal components analysis (PCA) in space with the Fourier transform in time~\cite{Chen:2012,Tu2014jcd,Kutz2016book,Brunton2019book}.  
Since its introduction, DMD has been rigorously connected to nonlinear dynamical systems~\cite{Rowley2009jfm}, as an approximation of the infinite dimensional Koopman operator~\cite{Koopman1931pnas,Koopman1932pnas}.  
The Koopman operator, introduced in 1931, provides an alternative operator-theoretic perspective for dynamical systems, and it has experienced a resurgence because of its utility for data-driven analysis~\cite{Mezic2004physicad,Mezic2005nd,Budivsic2012chaos,Mezic2013arfm,Bagheri2013jfm,Brunton2021koopman}. 

Since its introduction, DMD has been applied to a wide range of fields beyond fluid mechanics~\cite{schmid2011applications,Towne2018jfm}, including robotics~\cite{Berger2014ieee}, disease modeling~\cite{Proctor2015ih}, neuroscience~\cite{Brunton2014jnm}, finance~\cite{Mann2016qf}, and plasma physics~\cite{Taylor2018rsi,kaptanoglu2020characterizing}, to name only a few~\cite{Kutz2016book}.  
One of the key reasons for the broad success of DMD is its formulation in terms of linear algebra, statistical regression, and optimization, making it relatable to the recent work in machine learning~\cite{Brunton2020arfm} and highly extensible.  
Within a short time, DMD has seen dozens of algorithmic extensions, including for control~\cite{Proctor:2016DMDc,Kaiser2018prsa}, for large and streaming data sets~\cite{Hemati2014pof,Sayadi2016tcfd,Pendergrass2016arxiv,zhang2017online} incorporating ideas of randomized linear algebra~\cite{Bistrian2016ijnme,erichson2019randomized}, for model reduction~\cite{tissot2014model,Noack2016jfm}, for limited measurements~~\cite{Brunton2015jcd,Gueniat2015pof,Bai2017arxiv,Kramer2017siads} and time delays~\cite{Brunton2017natcomm,arbabi2017ergodic}, for multiresolution analysis~\cite{Kutz2016siads}, for selecting a sparse subset of dominant modes~\cite{Jovanovic2014pof}, and for Bayesian formulations~\cite{Takeishi2017JCAI}, among others~\cite{Azencot2019siads,Nair2020prf}.  
The PyDMD~\cite{Demo18pydmd} software package has emerged as an open source tool for several of these extensions.  

Two of the main avenues of DMD research have centered around noise robustness and closure issues related to modeling nonlinear systems.  
It was recognized early that DMD is quite sensitive to noise~\cite{Bagheri2014pof}, and several approaches have been proposed to address this sensitivity, including a forward-backward averaging~\cite{Dawson2016ef}, total least-squares regression~\cite{Hemati2017tcfd}, variable projection~\cite{Askham2018siads}, variational approaches~\cite{Azencot2019siads}, and robust principal component analysis~\cite{Scherl2020prf}. 
However, noise sensitivity is still a leading challenge, and is one that we directly investigate in this paper.  
Similarly, soon after the initial connection between DMD and Koopman theory~\cite{Rowley2009jfm}, there were several follow-up studies  addressing how and when a linear regression model can capture essential features of a nonlinear system.  
The extended DMD algorithm was developed to augment the DMD state with nonlinear observables to enrich the regression~\cite{Williams2015epl,Williams2015jcd,Li2017chaos}, and theoretical convergence results prove that in the infinite data limit, these models converge to the projection of the Koopman operator on this observable subspace~\cite{Korda2017arxiv}.  
However, often the observable subspace is not closed under action of the Koopman operator, so that it does not form an invariant subspace, resulting in poorly predicted dynamics~\cite{Brunton2016plosone}.  

In our own work we have sometimes encountered great difficulty in reconstructing time series with DMD. 
A particularly difficult class seems to be the reconstruction of \concept{intermediate transients} -- transient phenomena that decay faster than the slowest mode, but not much faster.
Here we investigate the possible sources of these difficulties by studying special class of systems for which both the latent linearization and the linearizing observables are completely known -- the class systems with linear latent dynamics which are observed via multinomial observables:
\begin{align}
  \mathbf{x}_{k+1} = \mathbf{A}  \mathbf{x}_k; && \mathbf{y} = P(\mathbf{x}) \label{eqn:gensys}
\end{align}
where $P(\mathbf{x})$ is a (typically low order) multinomial.
However, since we are considering models of physical systems, we will assume both system noise, and measurement noise:
\begin{align}
  \mathbf{x}_{k+1} = \mathbf{A} \mathbf{x}_k + (\text{system noise}); && \mathbf{y} = P(\mathbf{x}) + (\text{measurement noise}) \label{eqn:sys}. 
\end{align}

In this paper, we explore the combination of noise and nonlinearity in this controlled setting, where it is possible to isolate and analyze each effect with incrementally increasing severity.  
Specifically, we explore a particularly benign class of such systems, having 3-dimensional latent dynamics and observables that are monotone polynomials in individual coordinates.
We demonstrate that even for these very mildly nonlinear systems, observed over a subspace closed under the Koopman operator, DMD methods often fail to recover the spectrum and can have poor predictive value.
We also explore the influences of dataset metrics, the spectrum of the latent dynamics, the normality of the system matrix, and the geometry of the dynamics.
Based on results from this self-contained framework, we give generalizable recommendations regarding dataset properties for DMD analysis. 
All code used to generate these results are available as open source at: \url{}.

\section{Background}
Here we will introduce the dynamic mode decomposition (DMD), which is a recent technique for linear system identification.  DMD is closely related to Koopman operator theory, which provides a representation of a dynamical system in terms of the evolution its ``observables''.  We will also introduce Koopman theory, as it relates to performing DMD on data that is augmented with nonlinear observables.  

\subsection{Dynamic mode decomposition (DMD)}
Here we present the \emph{exact DMD} formulation of Tu et al.~\cite{Tu2014jcd}, along with several leading variants to handle noisy data~\cite{Hemati2017tcfd,Dawson2016ef,Askham2018siads}. 
Exact DMD is modified from the original algorithm of Schmid~\cite{Schmid2010jfm}, and both algorithms seek a best-fit linear operator $\mathbf{A}\in\mathbb{R}^{n\times n}$ that evolves a high-dimensional spatial measurement vector $\mathbf{x}\in\mathbb{R}^n$ forward in time:
\begin{align}
    \mathbf{x}_{k+1} = \mathbf{A}\mathbf{x}_k.
\end{align}
More specifically, DMD seeks the leading eigenvalues and eigenvectors of the matrix $\mathbf{A}$.  
The DMD eigenvectors $\boldsymbol{\phi}$, or \emph{modes}, have the dimension of the original data $\mathbf{x}$, and correspond to spatial structures that behave coherently in time according to the corresponding eigenvalue $\lambda$. 

The DMD algorithm begins with time-series data $\{\mathbf{x}_1, \mathbf{x}_2, \cdots , \mathbf{x}_m\}$, which is arranged into two data matrices:
\begin{align}
    \mathbf{X}=\begin{bmatrix}\mathbf{x}_1 & \mathbf{x}_2 & \cdots & \mathbf{x}_{m-1}\end{bmatrix} & & 
    \mathbf{X}' = \begin{bmatrix}\mathbf{x}_2 & \mathbf{x}_3 & \cdots & \mathbf{x}_m\end{bmatrix}.
\end{align}
The best-fit linear operator $\mathbf{A}$ may be solved for in the data equation 
\begin{align}
    \mathbf{X}' \approx \mathbf{A}\mathbf{X}
\end{align}
by a least-squares minimization:
\begin{align}
    \mathbf{A} = \text{argmin}_{\mathbf{A}}\|\mathbf{X}'-\mathbf{A}\mathbf{X}\|_F = \mathbf{X}'\mathbf{X}^\dagger
\end{align}
where $\|\cdot\|_F$ is the Frobenius norm and $\mathbf{X}^\dagger$ is the pseudoinverse of $\mathbf{X}$, computed via the singular value decomposition (SVD):
\begin{align}
    \mathbf{X}=\mathbf{U}\boldsymbol{\Sigma}\mathbf{V}^\T \quad\Longrightarrow\quad \mathbf{X}^\dagger = \mathbf{V}\boldsymbol{\Sigma}^{-1}\mathbf{U}^\T.
\end{align}
In practice, the SVD of $\mathbf{X}$ is typically truncated so that only $r\ll \min({m,n})$ modes are retained: $\mathbf{X}\approx \mathbf{U}_r\boldsymbol{\Sigma}_r\mathbf{V}^\T_r$.  

DMD was introduced for high-dimensional systems, such as fluid dynamics, where the state dimension $n$ may be in the millions.  
In this case, the matrix $\mathbf{A}$ may be intractably large, and instead of constructing $\mathbf{A}$ directly, it's low-dimensional representation $\tilde{\mathbf{A}}$ is analyzed in the $r$-dimensional subspace given by the columns of $\mathbf{U}_r$:
\begin{align}
    \tilde{\mathbf{A}} = \mathbf{U}_r^\T\mathbf{A}\mathbf{U}_r = \mathbf{U}_r^\T\mathbf{X}'\mathbf{V}_r\boldsymbol{\Sigma}_r^{-1}.
\end{align}
The eigenvalues of $\tilde{\mathbf{A}}$ are the same as the eigenvalues of $\mathbf{A}$, and are given by the diagonal elements of the matrix $\boldsymbol{\Lambda}$
\begin{align}
    \tilde{\mathbf{A}}\mathbf{W} = \mathbf{W}\boldsymbol{\Lambda}. 
\end{align}
The corresponding DMD modes, or eigenvectors of the high-dimensional system, are computed from the projected eigenvectors $\mathbf{W}$ as
\begin{align}
    \boldsymbol{\Phi}=\mathbf{X}'\mathbf{V}_r^\T\boldsymbol{\Sigma}_r^{-1}\mathbf{W}.
\end{align}
Tu et al.~\cite{Tu2014jcd} showed that the columns of $\boldsymbol{\Phi}$ are exact eigenvectors of the high-dimensional matrix $\mathbf{A}$. 

Once the DMD eigenvalues and eigenvectors are computed, it is possible to approximate the time series data through the DMD expansion, which is closely related to the Koopman mode decomposition:
\begin{align}
    \mathbf{x}_k \approx \sum_{j=1}^r\boldsymbol{\phi}_j\lambda^{k-1}b_j = \boldsymbol{\Phi\Lambda}^{k-1}\mathbf{b}.
\end{align}
The vector $\mathbf{b}$ contains the amplitudes of each DMD mode, and is often computed as
\begin{align}
    \mathbf{b} = \boldsymbol{\Phi}^\dagger \mathbf{x}_1.
\end{align}
Written in matrix form, the DMD expansion becomes
\begin{align}
    \mathbf{X} \approx 
    \underbrace{\begin{bmatrix}\vline & & \vline \\ \boldsymbol{\phi}_1 & \cdots & \boldsymbol{\phi}_r\\ \vline & & \vline\end{bmatrix}}_{\boldsymbol{\Phi}}
    \underbrace{\begin{bmatrix}b_1 & & \\ & \ddots & \\ & & b_r \end{bmatrix}}_{\text{diag}(\mathbf{b})}
    \underbrace{\begin{bmatrix}1 & \cdots & \lambda_1^{m-2}\\ \vdots & \ddots & \vdots \\ 1 & \cdots & \lambda_r^{m-2} \end{bmatrix}}_{\mathbf{T}(\boldsymbol{\Lambda})}.
\end{align}

The estimation of $\mathbf{b}$ above is based on the first snapshot of data only.  It is possible to improve this estimate with the following optimization over all snapshots:
\begin{align}
    \text{argmin}_{\mathbf{b}}\|\mathbf{X}-\boldsymbol{\Phi}\text{diag}(\mathbf{b})\mathbf{T}(\boldsymbol{\Lambda})\|_F.
\end{align}
Jovanovic et al.~\cite{Jovanovic2014pof} further add a sparsity promoting penalty to make the vector $\mathbf{b}$ as sparse as possible.  

Askham and Kutz~\cite{Askham2018siads} introduced the \emph{optimized} DMD algorithm that simultaneously optimizes over the modes and eigenvalues using variable projection.  In this approach, they combine the modes and amplitudes $\boldsymbol{\Phi}_{\mathbf{b}}=\boldsymbol{\Phi}\text{diag}(\mathbf{b})$ and solve the following optimization
\begin{align}
    \min_{\boldsymbol{\Lambda},\boldsymbol{\Phi}_{\mathbf{b}}}\|\mathbf{X}-\boldsymbol{\Phi}_{\mathbf{b}}\mathbf{T}(\boldsymbol{\Lambda})\|_F.
\end{align}
The optimized DMD approach has been proven to be quite robust to noisy data and data that is not perfectly approximated by a linear system.  

The sensitivity of DMD to noise was recognized early by Bagheri~\cite{Bagheri2014pof}, and several DMD variants have been proposed to improve the robustness of the algorithm.
Two leading approaches include the forward-backward (FB) DMD~\cite{Dawson2016ef} and the total least-squares (TLS) DMD~\cite{Hemati2017tcfd}.  
Both algorithms correct for sensor noise and process noise. 
Forward-backward DMD averages the standard DMD operator $\mathbf{A}$ and the inverse of a reverse-time DMD operator $\mathbf{A}^b$ computed by switching the order of $\mathbf{X}$ and $\mathbf{X}'$ in the algorithm.  The averaged operator 
\begin{align}
    \mathbf{A}_{\text{FB}}=\left(\mathbf{A}\left(\mathbf{A}^b\right)^{-1}\right)^{1/2}
\end{align} 
cancels out the bias introduced from noise.  
Similarly, total least-squares DMD makes note of the fact that if DMD is viewed as a regression from $\mathbf{x}_k$ to $\mathbf{x}_{k+1}$, then noise will affect both the dependent and independent variables.  
Thus, replacing the standard SVD-based least-squares regression with a total least-squares algorithm will improve the robustness to noise and remove bias.  In the TLS algorithm, the data $\mathbf{X}$ and $\mathbf{X}'$ are stacked into a larger matrix that is used to compute the low-dimensional subspace 
\begin{align}
    \mathbf{Z} = \begin{bmatrix} \mathbf{X}\\ \mathbf{X}'\end{bmatrix} = \mathbf{U_Z}\boldsymbol{\Sigma}_{\mathbf{Z}}\mathbf{V_Z}^\T.
\end{align}
The $\mathbf{U_Z}$ matrix is partitioned into four sub-matrices
\begin{align}
    \mathbf{U_Z} = \begin{bmatrix}\mathbf{U}_{\mathbf{Z},a}  & \mathbf{U}_{\mathbf{Z},b}  \\ \mathbf{U}_{\mathbf{Z},c} & \mathbf{U}_{\mathbf{Z},d} \end{bmatrix}
\end{align}
and the debiased estimate of the DMD operator $\mathbf{A}$ is given by 
\begin{align}
     \mathbf{A}_{\text{TLS}} = \mathbf{U}_{\mathbf{Z},c} \mathbf{U}_{\mathbf{Z},a}^{\dagger}.
\end{align}
In this work, we will compare the exact DMD algorithm, along with the forward-backward, total-least squares, and optimized DMD variants. 

\subsection{Koopman theory}\label{sec:koop}

DMD was quickly connected to nonlinear dynamical systems through Koopman operator theory~\cite{Rowley2009jfm}.  Here we provide a brief introduction of some important concepts that will be relevant when performing DMD on a state augmented with nonlinear observables.  

\newcommand{\kF}{\mathbf{F}}
Let $\kF:\mathbb{X} \times \R \to \mathbb{X}$ be the flow of a continuous- or discrete- time dynamical system, denoted $\mathbf{x}(t)=\kF^t(\mathbf{x}(0))$ where $\mathbf{x}(0), \mathbf{x}(t) \in \mathbb{X}$.
By ``flow'' we mean that $\kF$ is a representation of the group $(\R,+)$, i.e. $\forall t,s\in \R:\, \kF^t \circ \kF^s = \kF^{t+s}$ and $\kF^0$ is the identity map on $\mathbb{X}$.

An ``observable'' is a map $\mathbf{g}:\mathbb{X} \to \mathbb{V}$ into some space of observation values $\mathbb{V}$.
We will usually assume $\mathbb{V}$ is of the form $\C^d$: some vector space over the field of complex numbers.

The ``Koopman Operators'' (a.k.a. ``composition operators'') $\opK^t$ act on observables (e.g. $\mathbf{g}$) by:
\begin{equation}
  \forall t\in\R,\, \mathbf{x}\in\mathbb{X}:\, (\opK^t \mathbf{g}) (\mathbf{x}) = \mathbf{g}( \kF^t (\mathbf{x}) ).
\end{equation}

Our interest lies in finding observables $\mathbf{g}$ such that $\mathbf{g} \circ \kF^t$ is a linear system in $\mathbb{V}$, i.e. there exists a matrix $\M \in \Aut(\mathbb{V})$ such that $\exp(t \M) \cdot \mathbf{g} = \mathbf{g} \circ \kF^t = \opK^t \mathbf{g}$.
Such $\mathbf{g}$ spans an ``eigenspace'' of $\opK$ with a finite-dimensional evolution operator $\M$.
In the case where $\M$ is a scalar, this coincides with the familiar definition of eigenvectors of a linear operator.

\newcommand{\Ma}{\mathbf{A}}
\newcommand{\Mb}{\mathbf{B}}
Note that for any two eigenfunctions $g_{_\Ma}, g_{_\Mb}$ with commuting (matrix-valued) eigenvalues $\Ma$ and $\Mb$, the function $g(\mathbf{x}) := g_{_\Ma}(\mathbf{x})\cdot g_{_\Mb}(\mathbf{x})$ satisfies
\begin{equation}
  g \circ \kF^t = (g_{_\Ma}\circ \kF^t) \cdot (g_{_\Mb}\circ \kF^t)
    = \exp(t \Ma) \cdot g_{_\Ma} \cdot\exp(t \Mb) \cdot g_{_\Mb}
    = \exp(t (\Ma+\Mb)) \cdot (g_{_\Ma} \cdot g_{_\Mb})
    = \exp(t (\Ma+\Mb)) \cdot g,
\end{equation}
and thus the pointwise product of eigenfunctions whose eigenvalues commute is itself an eigenfunction whose eigenvalues are the sum of the constituent eigenvalues.
This also implies (trivially) that non-negative integer powers of eigenfunctions are eigenfunctions.

Together these observations imply that any monomial of eigenfunctions $g_1^{\alpha_1}\cdots g_m^{\alpha_m}$ whose respective matrix eigenvalues were $\M_k$ for $g_k$, and for whom all $\M_k$ commute, is itself an eigenfunction which has the matrix eigenvalue $\alpha_1\M_1 +\ldots+\alpha_m\M_m$.

It is also possible to compute DMD on an augmented state that includes nonlinear functions of the measured state $\mathbf{x}$, in a procedure called \emph{extended DMD}~\cite{Williams2015epl,Williams2015jcd,Li2017chaos}.  
By including nonlinear functions, it is possible to approximate the projection of the Koopman operator onto a larger space of functions, where it may be possible to better approximate the relevant eigenfunctions.  
It was recently shown that extended DMD is equivalent to the earlier \emph{variational approach of conformation dynamics} (VAC)~~\cite{noe2013variational,nuske2014jctc,nuske2016variational}, introduced by No\'e and N\"uske to simulate molecular dynamics with a broad separation of timescales.  
Further connections between eDMD and VAC and between DMD and the time lagged independent component analysis (TICA) are explored in a recent review~\cite{klus2017data}.  
This work will explore extended DMD where the state is augmented with monomial functions.

\section{Methods}
Generic stable linear systems can exhibit a great many behaviors. 
To explore several of these, including some of the interactions between different eigenvalues and eigenspace geometries, we constructed a 3-dimensional linear system.

\subsection{A ``simple'' linear system}

We created a stable real linear system, with states $\mathbf{x} \in \R^3$, and considered specifically the case of a spectrum that consists of a (stable) complex-conjugate pair and a single (stable) real eigenvalue.
Such a system has two invariant subspaces -- a two-dimensional subspace in which the complex-conjugate pair generate a spiral sink, and a one-dimensional subspace in which the real eigenvalue induces exponential convergence.
We chose, without loss of generality, to make the two dimensional invariant space the XY plane with the major and minor axes of the ellipsoidal spiral to be aligned with the X and Y axes.
Thus, the most general 3D system matrix $\mathbf{A}$, with a spectrum consisting of one complex-conjugate pair $\alpha\pm i \beta$ and one real eigenvalue $\lambda$, can be constructed as follows.
We took $\mathbf{A}$ to be constructed from an eigenvector matrix $\mathbf{Q}$, a diagonal matrix $\mathbf{S}$ to rescale the major vs the minor axis in the XY plane, and a block diagonal eigenvalue matrix $\mathbf{\Lambda}$,
\begin{equation}\label{eqn:A}
\mathbf{A} := \mathbf{QS \Lambda} \mathbf{S}^{-1}\mathbf{Q}^{-1}.
\end{equation}
We parameterized $\mathbf{Q}$, $\mathbf{S}$, $\mathbf{\Lambda}$ by
\begin{equation}\label{eqn:QSL}
\mathbf{Q}:=\left[\begin{array}{ccc}
1 & 0 & \sin(\theta)\cos(\phi) \\
0 & 1 & \sin(\theta)\sin(\phi) \\
0 & 0 & \cos(\theta)
\end{array}\right],
\mathbf{S} := \left[\begin{array}{ccc}
s & 0 & 0 \\
0 & 1 & 0 \\
0 & 0 & 1
\end{array}\right],
\mathbf{\Lambda} := \left[\begin{array}{ccc}
\alpha & \beta & 0 \\
-\beta & \alpha & 0 \\
0 & 0 & \lambda
\end{array}\right].
\end{equation}
The parameters $\phi$, $\theta$, and $s$ exhaust all possible 3D linear dynamics with the spectrum of $\mathbf{\Lambda}$ up to an orthogonal similarity transform, as proved in \ref{sec:proof}.

\subsection{Nonlinear observations of the system}

We took as our nonlinear observations of the system the output
\begin{align}
  \mathbf{y} := [x_1+0.1(x_1^2+x_2 x_3),~ x_2+0.1(x_2^2+x_1 x_3),~ x_3+0.1(x_3^2+x_1 x_2)]^\T.
  \label{eqn:y}
\end{align}
For the range of $\mathbf{x}$ values we considered (initial conditions on the unit sphere), $\mathbf{y}$ coordinates were strictly monotone in the corresponding $\mathbf{x}$ coordinates, and observations corresponded to unique states.
As \refFig{traj} illustrates, the trajectories of $\mathbf{y}$ and of $\mathbf{x}$ are so similar that for a casual observer they would seem nearly indistinguishable.
However, as we will show in \ref{sec:results}, these nonlinearities can have profound impact on our ability to model the system using DMD.

\begin{figure}[t]\label{fig:traj}
  \centering
  \includegraphics[width=0.9\textwidth]{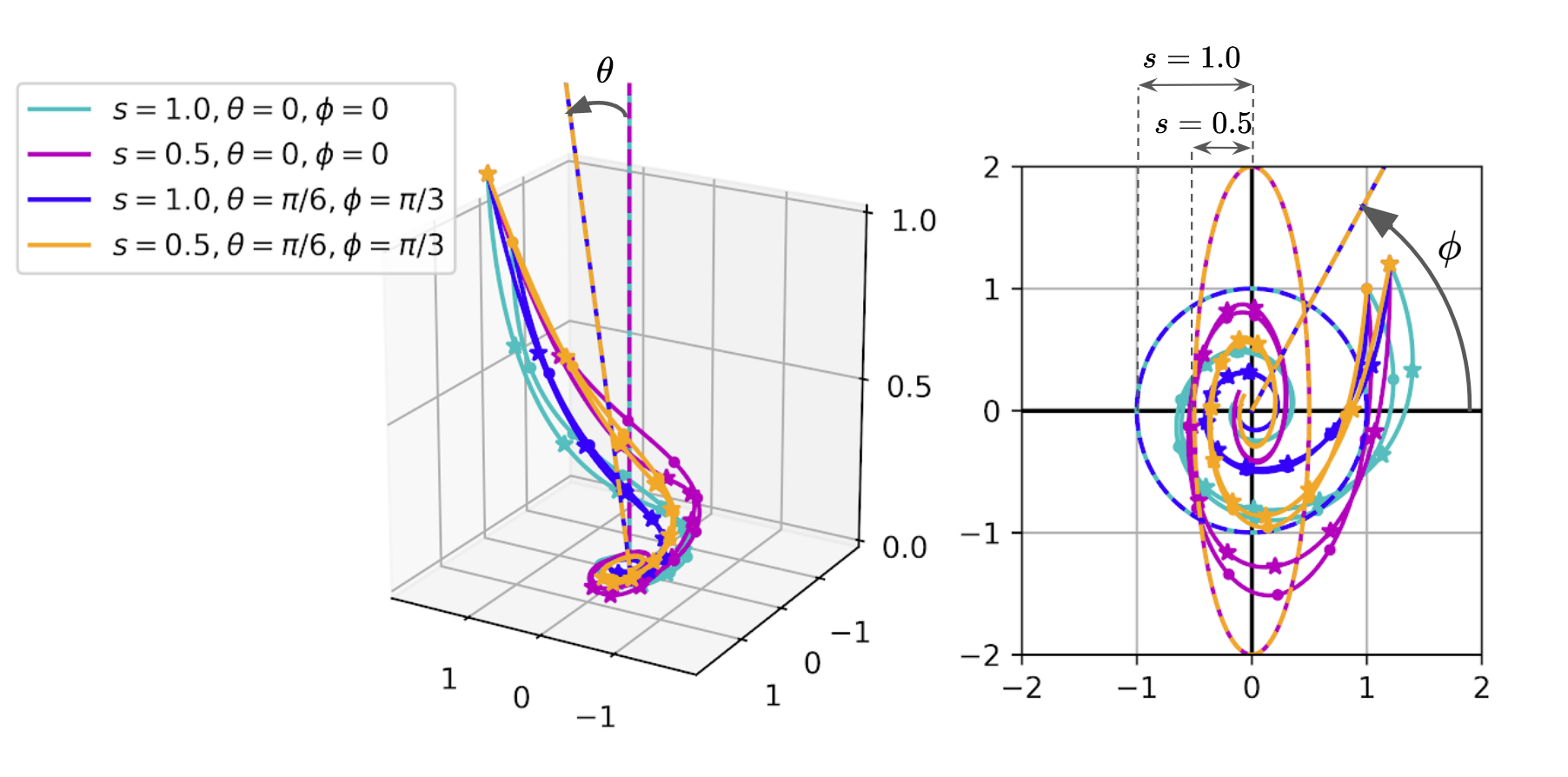}
  \caption{Example trajectories of our 3D system \refEqn{sys} for various parameters of \refEqn{QSL} (colors). %
  Eigenvalues were $0.9(\sqrt{3}/2 \pm 0.5i)$ and $0.6$ throughout, and we used the same initial condition. %
  We plotted trajectories without noise added for both the latent $\mathbf{x}$ (dots), and the mildly nonlinear observable $\mathbf{y}$ (\refEqn{y}; stars), in 3D perspective (left) and in an XY plane projection (right). %
  In the 3D plot we indicated the direction of the $0.6$ eigenspace controlled by $\theta$ and $\phi$ and given by the last column of $Q$ in \refEqn{QSL} (dashed line). %
  We showed the effects of $s$ in the $XY$ plane using ellipsoids (dashed line) whose major and minor axes have ratio $s^2$. %
  The trajectories demonstrate that our nonlinear observations are quite similar to the latent linear dynamics, and that trajectories vary their geometry with the parameters of \refEqn{QSL} even while the spectrum remains the same.
  }
\end{figure}

Because $\mathbf{A}$ had distinct eigenvalues, it was diagonalizable over $\C$.
Consider the diagonal matrix $\mathbf{\tilde\Lambda}$ with 9 diagonal elements comprising the spectrum of $\mathbf{A}$ and all 6 pairwise sums of eigenvalues of $\mathbf{A}$.
Every monomial of the form $x_1^{k_1} x_2^{k_2} x_3^{k_3}$ with $1 \leq k_1+k_2+k_3 \leq 2$ evolves exponentially over time as one of the diagonal elements of $\mathbf{\tilde\Lambda}$.
We conclude that the dynamics of $\mathbf{y}$ could just as well have been written as $\mathbf{z}_{n+1} := \mathbf{P} \mathbf{\tilde\Lambda} \mathbf{P}^{-1} \mathbf{z}_n$, $\mathbf{y} := \mathbf{W} \mathbf{z}$ for matrices $\mathbf{P}$ and $\mathbf{W}$.

This is of utmost importance, since it implies that \emph{this nonlinearly observed 3D linear system is exactly reproduced in a 9D linear system observed linearly}.
More generally, all monomial observables of order $m$ or less can be represented using eigenspaces whose eigenvalues can be obtained as a sum of no more than $m$ of the original eigenvalues; see \ref{sec:diagsys} for details.
We may therefore rest assured that a DMD method that allowed for 9 modes could in fact predict the values of $\mathbf{y}$ to without any truncation errors due to the nonlinearity of the observations.

\subsection{Dataset preparation}\label{sec:dataset}

Our numerical experiments used the 3D systems, with Gaussian system noise and Gaussian measurement noise both set to zero mean and a standard deviation of $0.05$.
The initial conditions were chosen uniformly distributed on the unit sphere.

We explored all combinations of $s \in \{0.1, 0.5, 1.0 \}$, $\phi \in \{0, 0.79, \pi/2\} $, $\theta \in \{0, 1.1, 1.2, 1.3, 1.4, 1.5,1.52,1.53,1.560\} $ with a fixed eigenvalue spectrum $0.5(\sqrt{3}/2 \pm 0.5i)$ and $0.8$. 
We also explored different spectra with complex conjugate eigenvalue pair $0.5(\sqrt{3}/2 \pm 0.5i)$ and a real eigenvalue varying from 0.01 to 1, with orthogonal eigenvector space $s=1, \theta=0$. 
Furthermore, we considered datasets consisting of $100$ points each with: (i) $N=50$ initial conditions for trajectories of length $L=2$; (ii) $N=10$ initial conditions for trajectories of length $L=10$; (iii) $N=2$ initial conditions for trajectories of length $L=50$.
Because of the presence of system noise, all of these datasets meet the condition of \concept{persistence of excitation}~\cite{pe1985ieee}.

In the following, we present contour plots showing the distribution of eigenvalue estimates.
Wherever we show such plots, they smoothed with a Gaussian kernel of $\sigma=0.01$.
We first sorted all eigenvalues by their imaginary parts, and then registered each to its nearest true eigenvalue bin.
To ensure the spectrum symmetry, we discarded estimates with more real eigenvalues than the true spectrum. 
If there were any extra complex eigenvalues pairs, we registered two copies of the average value.
Furthermore, we added data to these contour plots in batches of $300$ numerical experiments. 
This number was sufficient so that the KL-divergence in $P_1$ and $P_2$, the probability distributions before and after a new estimated eigenvalue $\hat{\lambda}$ was added, was smaller than 0.001 for all experiments, i.e. $\sum_{\hat{\lambda}}P_1(\hat{\lambda}) \log\left(\frac{P_1(\hat{\lambda})}{P_2(\hat{\lambda})}\right) < 0.001$.

\section{Results}\label{sec:results}
In this section, we present the eigenvalue estimation by exact DMD on a 3D linear system with mild nonlinear measurements and observation noise.
We examine the effects of system normality, spectrum, mild nonlinearity in the observables, and varying number of initial conditions through their effect on the recovered spectrum and the prediction of the dynamics.  
We found that mildly nonlinear observations and slight non-normality caused large eigenvalue estimation errors.
For a system with a well-posed system matrix, having a dataset with more initial conditions and shorter trajectories can significantly improve the prediction. 
With a slightly ill-conditioned system matrix, a moderate trajectory length improves the spectrum recovery. 

\subsection{Fixed spectrum with different system matrix normality}
We fixed the spectrum of the system matrix to be $0.5(\sqrt{3}/2 \pm 0.5i)$ and $0.8$, and varied its $\phi$, $s$, and $\theta$ as described in \refEqn{A}. 
The linear states evolved with Gaussian system noise $\sigma=0.05$, and they were observed via monomials with Gaussian observation noise $\sigma=0.05$. 
In \refFig{eigDensity1-cond} and \refFig{eigDensity2-cond}, we show the DMD eigenvalue density plot under five different normality settings, using first order and second order monomial observables. 
In the first column, the system matrix $\mathbf{A}$ is well-conditioned, with orthogonal eigenvectors and a centrosymmetric spiral sink. 
In the second column, the  eigenvectors remain orthogonal, and we stretch the spiral sink by varying the parameter $s$, so that the resulting trajectory converges following an ellipsoidal spiral with a minor-to-major axes ratio of $s^2$.
The eigenvalue density contour is stretched vertically when $s$ decreases, which yields a larger variance in the complex frequency estimation.
In the third and fifth column, we varied the angle between the 1-dimensional stable subspace corresponding to the real eigenvalue and the 2-dimensional subspace corresponding to the spiral sink, with the spiral sink remaining centrosymmetric. 
The fourth column exhibits both non-orthogonal eigenvectors and a non-centrosymmetric spiral sink. 
When eigenvectors are non-orthogonal, a complicated multi-modal error structure shows up in the $L=10,50, \theta=1.40, 1.56$ cases. 

To thoroughly explore the effect of system matrix normality on spectrum recovery and dynamics prediction, in \refFig{std-cond1} and \refFig{std-cond2} we plotted the standard deviation of each estimated DMD eigenvalue versus the condition number. 
The complex eigenvalues always come out as conjugate pairs from exact DMD, so we give one plot per pair. 
The discarded trials plot (right lower corner in \refFig{std-cond1} and \refFig{std-cond2}) gives the percentage of DMD estimates resulting in a structurally incorrect estimation, i.e. where the number of real eigenvalues is larger than the ground truth system, because it breaks the symmetry pattern of the complex conjugate pairs. 
When the condition number is below 100, the eigenvalue estimation from a dataset with more initial conditions and short trajectories (squares) has lower std, and less discarded trials. 
When the condition number is larger than 100, the cases with moderate trajectory length (circles) performs better. 

The estimation error caused by system matrix non-normality can be amplified through higher order observables. In the first two columns of \refFig{eigDensity1-cond}, we find that the eigenvalue contours are close to the true eigenvalues. 
The second order observables in \refFig{eigDensity2-cond} amplified the effect of slight changes in the condition number of the system matrix.
Only the slowest modes (red and purple contours) were identifiable. 
Furthermore, in \refFig{std-cond2}, the standard deviation of the eigenvalue estimates increases to 0.25 for most eigenvalues with a modest condition number of 4. 

\begin{figure}\label{fig:eigDensity1-cond}
  \centering
  \includegraphics[width=0.9\textwidth]{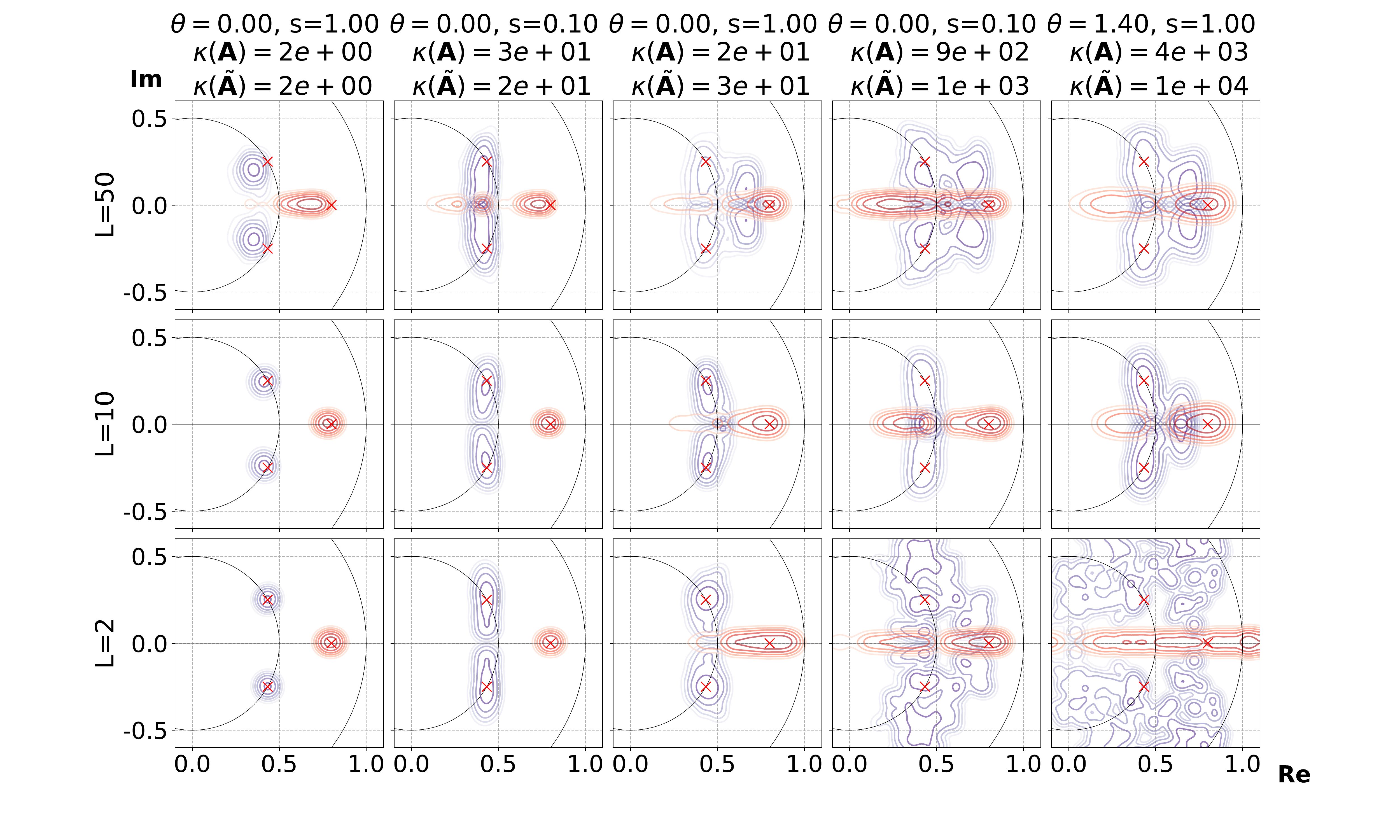}
  \vspace{-0.3in}
  \caption{Eigenvalue density plot with first order linear observables.\
    Different rows have different trajectory lengths and number of trajectories, while holding the same total data points.\  
    The system $\mathbf{A}$ matrix is constructed by $\phi=0$ and $\theta, s$ specified on top of each column.\
   $\kappa(\mathbf{A})$ is the condition number of ground truth system matrix, and $\kappa(\Tilde{\mathbf{A}})$ is the mean of estimated condition number found by exact DMD.\
  Red crosses are ground truth eigenvalues.}
\end{figure}

\begin{figure}\label{fig:std-cond1}
  \centering
  \includegraphics[width=\textwidth]{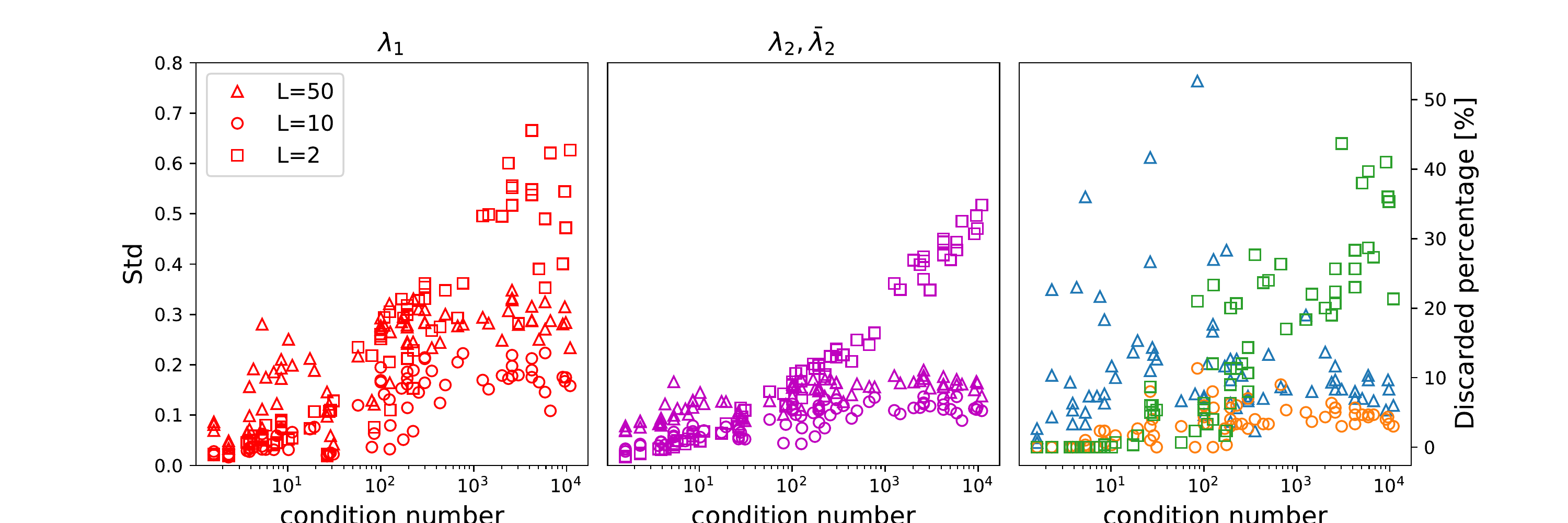}
  \vspace{-0.1in}
  \caption{The standard deviation of each estimated eigenvalue versus system $\mathbf{A}$ matrix condition number plot.\
   The colors of different eigenvalues match the color of contours in \refFig{eigDensity1-cond}.}
\end{figure}

\begin{figure}\label{fig:eigDensity2-cond}
  \centering
  \includegraphics[width=0.9\textwidth]{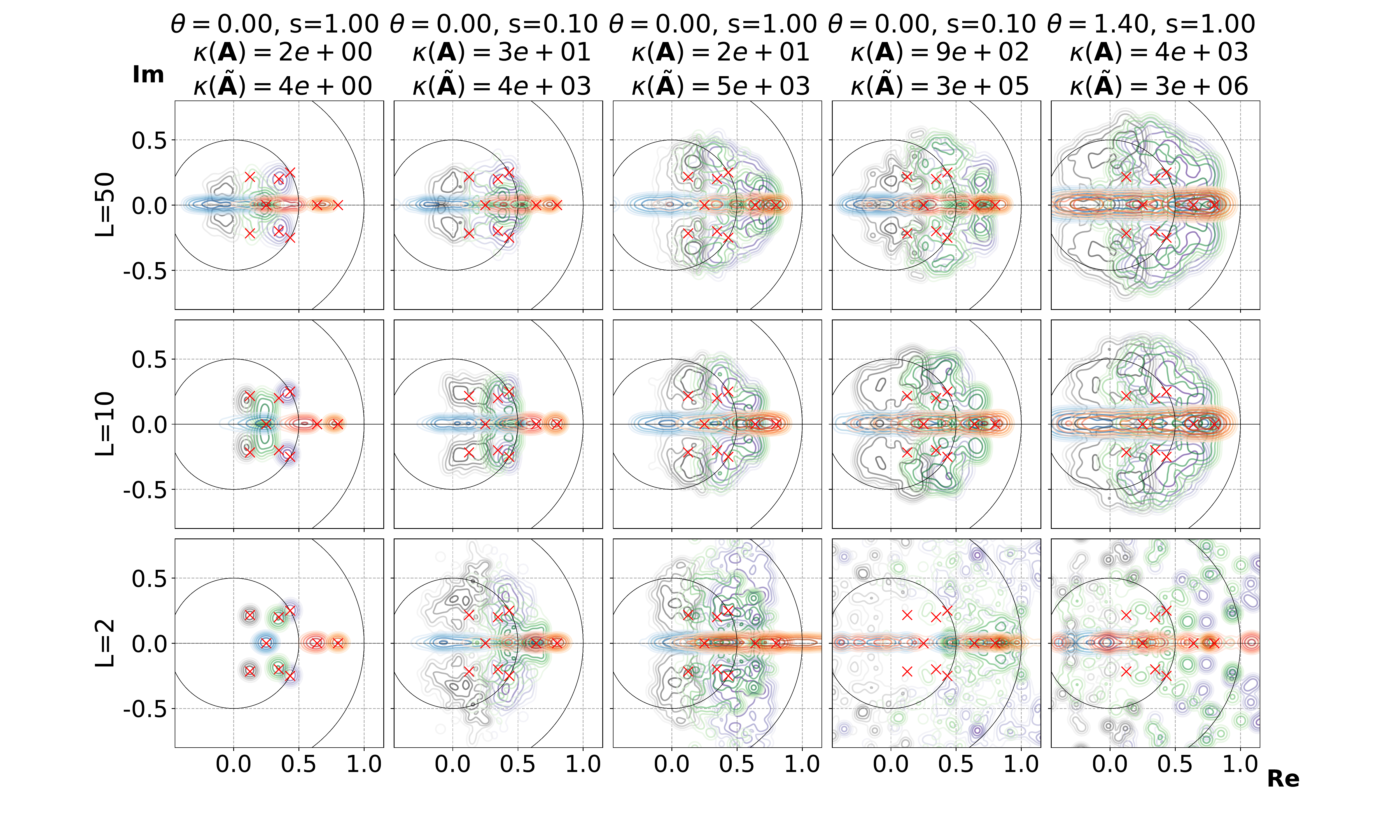}
  \vspace{-0.3in}
  \caption{Eigenvalue density plot with monomial observables up to 2nd order.\
    Different rows have different trajectory lengths and number of trajectories, while holding the same total data points.\
    The system $\mathbf{A}$ matrix is constructed by $\phi=0$ and $\theta, s$ specified on top of each column.\
    $\kappa(\mathbf{A})$ is the condition number of ground truth system matrix, and $\kappa(\Tilde{\mathbf{A}})$ is the mean of estimated condition number found by exact DMD.\
    Red crosses are ground truth eigenvalues.}
\end{figure}

\begin{figure}\label{fig:std-cond2}
  \centering
  \includegraphics[width=\textwidth]{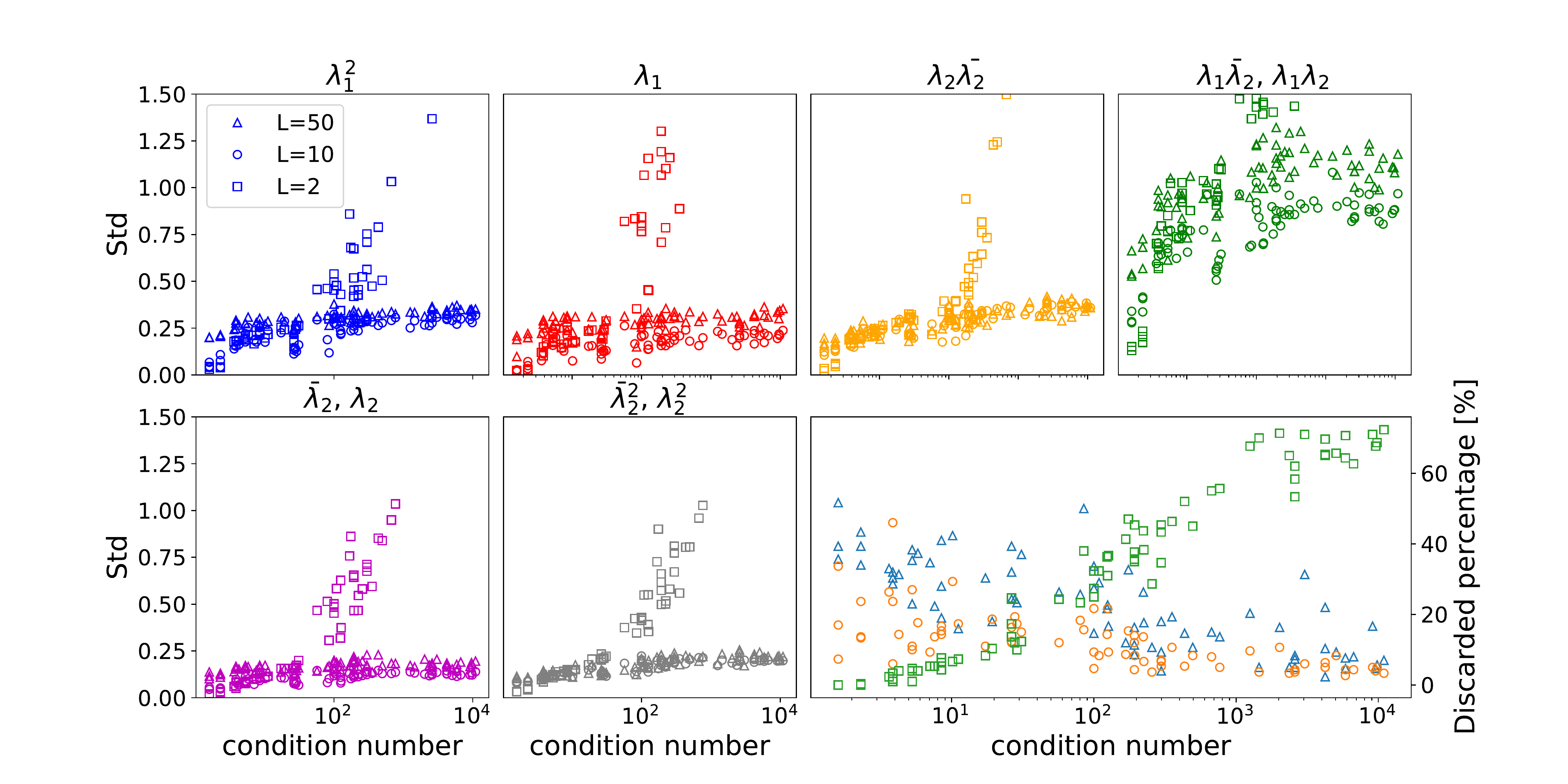}
  \vspace{-0.3in}
  \caption{The standard deviation of each estimated eigenvalue versus system $\mathbf{A}$ matrix condition number plot.\
   The colors of different eigenvalues match the color of contours in \refFig{eigDensity2-cond}. }
\end{figure}

\subsection{Different spectrum with orthogonal eigenspace}
In this section, we kept an orthogonal eigenvector space with $s=1, \theta=0$ and a complex conjugate eigenvalue pair $0.5(\sqrt{3}/2 \pm 0.5i)$. We varied the real eigenvalue from 0 to 1.
The linear states evolved with Gaussian system noise $\sigma=0.05$, and were observed via monomials having Gaussian observation noise $\sigma=0.05$. 
In \refFig{eigDensity1-eigmoving}, we show the DMD eigenvalue density plot under three different spectra using data prepared as described in \ref{sec:dataset}. In \refFig{std-eig1} and \refFig{std-eig2}, we show the standard deviation of the estimated DMD eigenvalues versus varying the eigenvalue $\lambda$. 

More initial conditions with shorter trajectories uniformly outperforms less initial conditions with longer trajectories, resulting in more accurate and precise estimations.
In the case of second order monomial observables with a trajectory length of L=2, the standard deviation of the estimated eigenvalues increases when $\lambda$ moves closer to the conjugate pair.
This suggests a dataset with short trajectories might be more sensitive to a small spectral distance.
Nonetheless, this case still has lower std comparing to less initial conditions with longer trajectories.
When $\lambda$ moves from 0.2 to 1.0, the percentage of discarded trials for L=2 decreases and stays at 0, for L=10 this percentage remains about the same, and for L=50 it increases to 50\%. 

\begin{figure}\label{fig:eigDensity1-eigmoving}
  \centering
  \includegraphics[width=\textwidth]{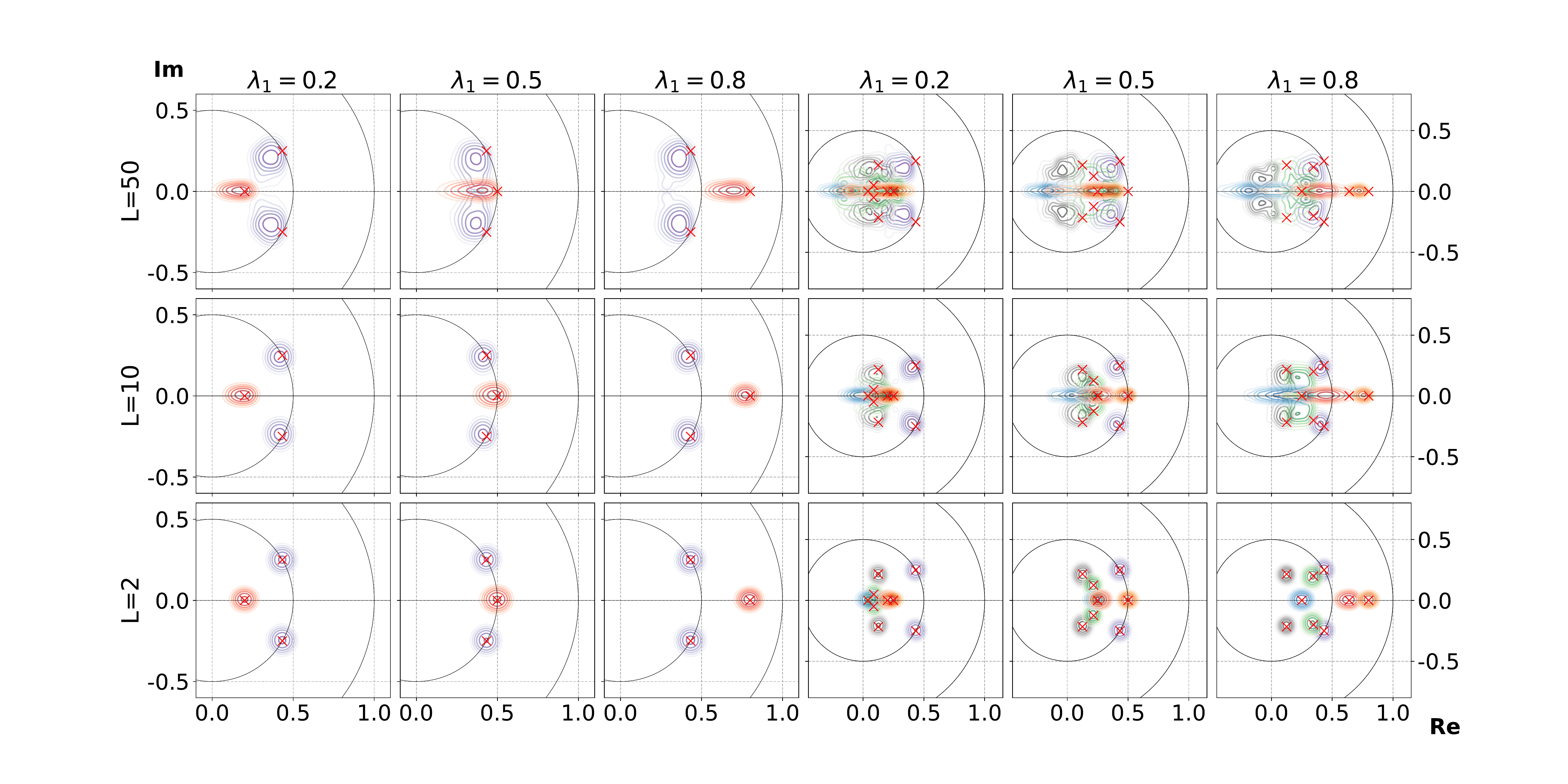}
  \vspace{-0.3in}
  \caption{Eigenvalue density plot with linear observables (left 9 plots) and observables up to second order monomials (right 9 plots). \
  Different rows have different trajectory length, and different columns have different real eigenvalue.\
  Red crosses are ground truth eigenvalues.}
\end{figure}

\begin{figure}\label{fig:std-eig1}
  \centering
  \includegraphics[width=\textwidth]{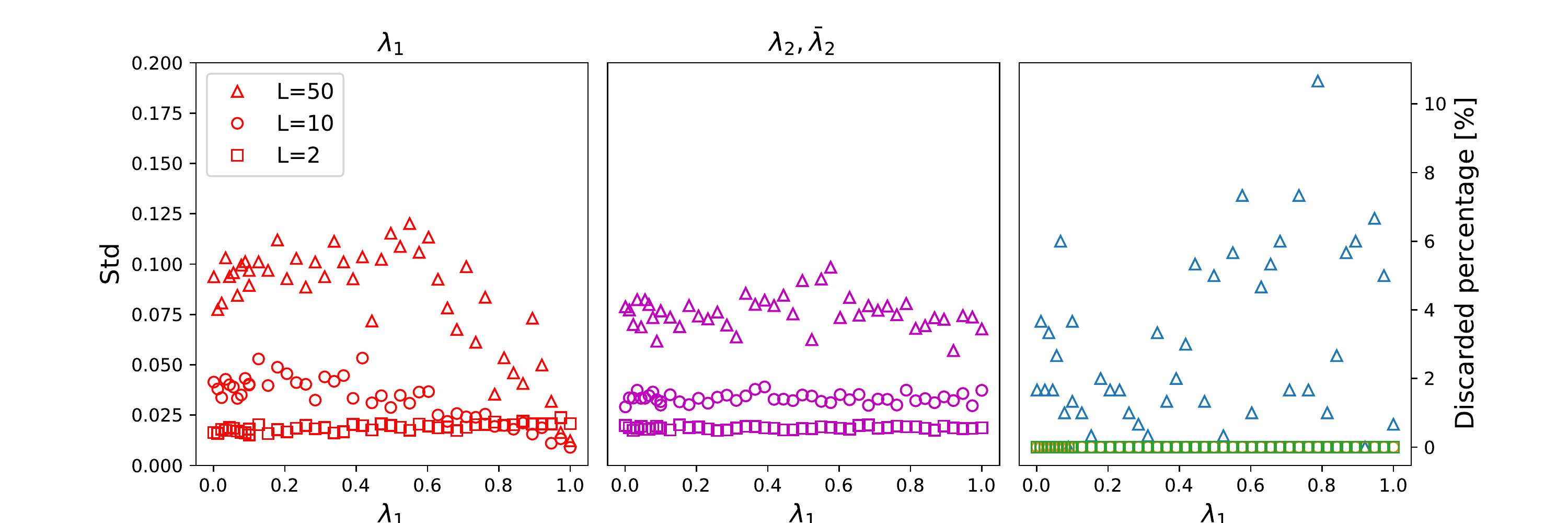}
  \vspace{-0.1in}
  \caption{The standard deviation versus real eigenvalue.\
   The colors of different eigenvalues match the color of contours in the left 9 plots of \refFig{eigDensity1-eigmoving}. }
\end{figure}

\begin{figure}\label{fig:std-eig2}
  \centering
  \includegraphics[width=\textwidth]{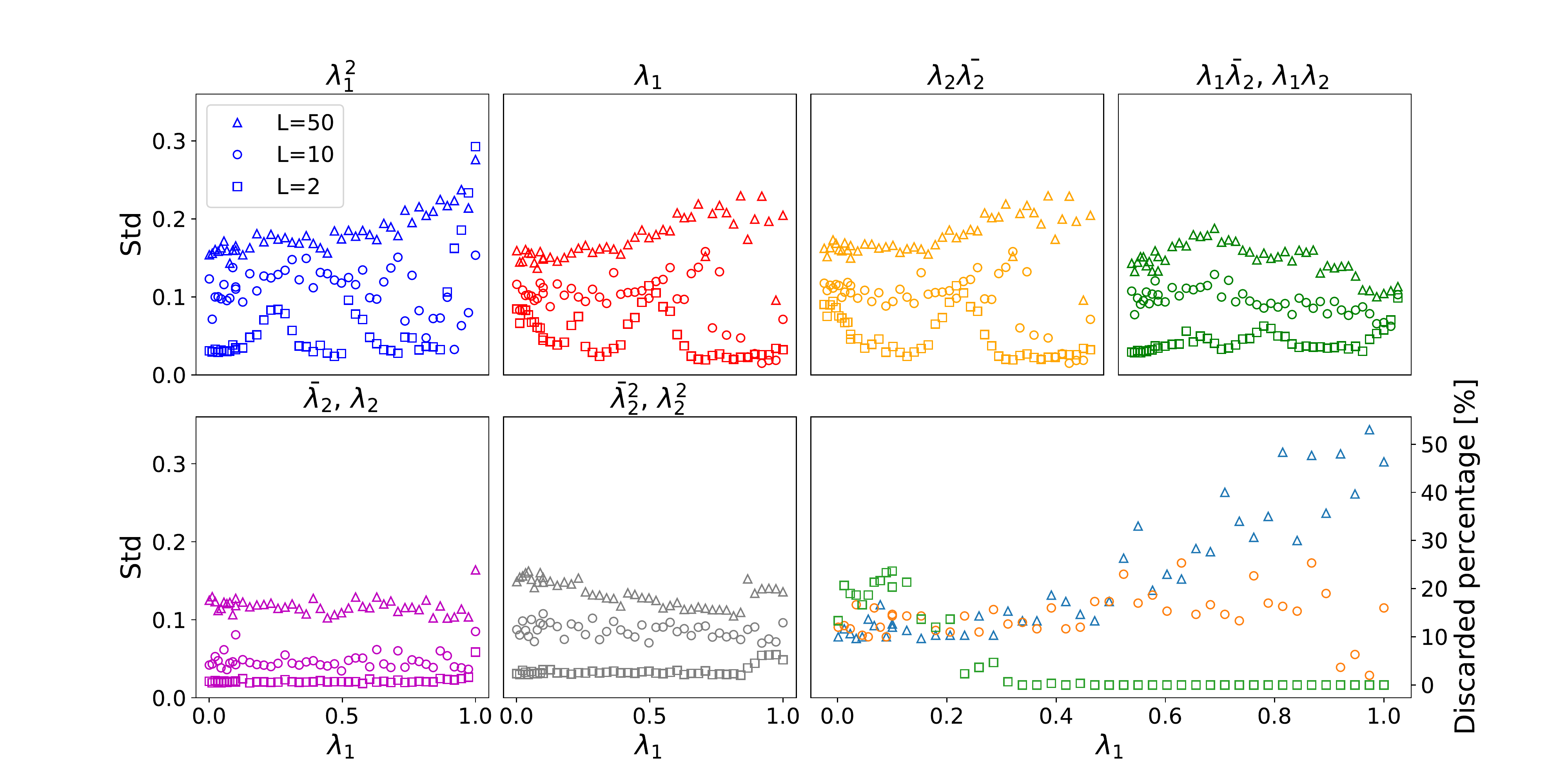}
  \vspace{-0.3in}
  \caption{The standard deviation versus real eigenvalue.\
   The colors of different eigenvalues match the color of contours in the right 9 plots of \refFig{eigDensity1-eigmoving}. }
\end{figure}
\section{Discussion and Conclusion}

We have shown that under conditions of mild system and observation noise, and mild nonlinearity, DMD can encounter significant difficulties in recovering a linear model with the correct spectrum and an accurate prediction of the dynamics, even when an exact linear model is known to exist. 
To study these effects, we constructed a self-contained 3D linear system where it is possible to independently control the eigenvalue spectrum, the normality of the system matrix, the nonlinearity of the observations, the system and observation noise, and the number and length of the trajectories used as training data.  

In this study, we find that DMD is quite sensitive to non-normality (see figs. \ref{fig:eigDensity1-cond}, \ref{fig:std-cond1}, \ref{fig:eigDensity2-cond}, \ref{fig:std-cond2}, and supplemental figs. \ref{fig:app-3}, \ref{fig:app-4}), and to the existence of higher order nonlinear terms (see figs. \ref{fig:eigDensity2-cond}, \ref{fig:eigDensity1-eigmoving} and supplemental figs. \ref{fig:app-2}, \ref{fig:app-4}).
As to the question of which data-set metrics govern this sensitivity, in some cases shorter data with more initial conditions performs better (as in figs. \ref{fig:eigDensity1-eigmoving}, \ref{fig:std-eig1}, \ref{fig:std-eig2}), and in some data fewer, longer time series produce better results (see \ref{fig:eigDensity1-cond} \ref{fig:eigDensity2-cond}).  
The challenge of non-normality has implications for the identification of fluid systems, as highly sheared flows are generally characterized by non-normal linearized dynamics, which was an original motivation for DMD~\cite{Schmid2010jfm}.  

We first fixed the spectrum of our linear system, varying the normality of system matrix and the geometry of trajectories. 
The results suggest that an ill-condition system matrix and mildly-nonlinear observables cause DMD to fail to recover the full spectrum; this phenomenon is especially evident in \refFig{eigDensity2-cond}. 
Furthermore, DMD estimation can be very sensitive to noise when the angle between eigenspaces is small. 
For the case of linear observables, when we decrease the angle between invariant subspaces, a multi-modal error structure appears in \refFig{eigDensity1-cond}. 
The complex eigenvalue pair and the real eigenvalue appear to have switched real parts. 
The cause of this error structure is not understood and remains an interesting avenue of future work. 

We then explored systems with different eigenvalue spectra, where we changed the distance between the eigenvalues. For short trajectories ($L=2$), closer eigenvalues result in higher standard deviation in the predicted spectrum, using second order observables. 
Generally, with enough initial conditions, DMD yields descent prediction of the spectra for all well-conditioned systems explored here. 

DMD under second order observations generally does not provide good estimations for transient oscillatory modes. 
To ensure this is not the effect of spectrum resonance caused by first and second order monomial observations, we compute the same contour plot on a 9D non-resonant linear system with similar spectrum in \refFig{appendix-nonresonant}. 
The result is similar to the right 9 plots of \refFig{eigDensity1-eigmoving}. 
Further, we show that using the latest denoising DMD algorithms in \refSec{appendix-algorithms}, the difficulty in spectrum recovery often persists, although each method performs well for certain cases. 

The 3D linear system analyzed here is a special case of a larger class of systems designed specifically to test the sensitivity of DMD to the challenges listed above.  
The class of diagonal linear systems with multinomial observables provides an algebraically convenient yet sufficiently general class of systems to use to study these problems.
As the theorems in appendix sections \ref{sec:dense} and \ref{sec:diagsys} show, diagonal linear systems with multinomial observables are universal in the sense of being able to approximate any smooth, stable dynamical system (details in Theorem \ref{thm:approx}). 
Furthermore, these systems have a finite dimensional linear representation which is exact -- a linear observation of a linear dynamical system (see Theorem \ref{thm:flatten}) -- and computable in close form.
Thus, the class of diagonal linear systems with multinomial observables is a useful space to test any technique that identifies linear models, including the many variants of DMD.

There are several future directions that are motivated by this work.  
It will be important to explore how these challenges scale to higher dimensional systems.  
There is also the potential to design sampling strategies to improve the conditioning of the DMD procedure. 
In addition to understanding and characterizing these challenges on numerical examples, it will be interesting to apply these careful studies to low-dimensional mechanical systems.  

\section{Acknowledgements} The authors acknowledge funding support from the Army Research Office (W911NF-17-1-0306). 

\section{References}
\begin{spacing}{.8}
 \small{
 \setlength{\bibsep}{4.1pt}
\bibliographystyle{IEEEtran}
\bibliography{main.bib}
 }
\end{spacing}
 
\newpage
\appendix
\newcommand{\Xfx}{\mathfrak{X}^\infty_\text{fix}}
\newcommand{\Gfx}{\mathcal{G}_\text{fix}}
\newcommand{\Psinv}{\Psi^{-1}_\varepsilon}
\newcommand{\xF}{\mathbf{x}}
\newcommand{\yF}{\mathbf{y}}
\newcommand{\xDom}{\mathbb{X}}
\theoremstyle{plain}
\newtheorem{Thm}{Theorem}

\section{Appendix}\label{sec:appendix}
At first, it might appear that the class of systems described by multinomial observations of diagonal linear dynamics is rather restrictive. 
As we show in Section \ref{sec:dense}, the converse is true -- these systems are quite general.
Furthermore, we show in Section \ref{sec:diagsys} that such systems can be viewed instead of as a multinomial observation of a non-resonant linear system, as a linear observation of an extended linear systems whose eigenvalues fall on a lattice.
Thus a system that identifies linear observations of linear dynamics is actually quite general.

\subsection{Density of linearizable systems}\label{sec:dense}
Consider a physical model in the form of a finite dimensional ordinary differential equation $\dot \xF = F_0(\xF)$, with $x\in\xDom$, and the ``flow'' $\kF_0$ associated with $F_0(\cdot)$ as defined in \S \ref{sec:koop}.

The vast majority of control theory as used in practical engineering applications, and a large part of the mathematical modeling done in the physical sciences, deals with the behavior of such systems around equilibria, i.e. points of the form $F_0(\xF_0)=0$.
Without loss of generality, we shall assume that the equilibrium of interest is at $0$, i.e. $F_0(0)=0$.
Furthermore such models are usually of stable systems, i.e. have a neighborhood $U$ in which for all $\xF\in U$, $\lim_{t\to\infty}\kF_0^t(\xF)=0$ for the flow $\kF_0$.
We will restrict our discussion to a compact neighborhood $U$ of $0$ for which the dynamics are stable.

Borrowing from the definitions in \cite{kvalheim2020generic}, we define $\Xfx(U)$ as the set of all $\Cont^\infty$ smooth vector fields with a stable fixed point at $0$ with stability basin $U$.
Within this set there exists $\Gfx(U) \in \Xfx(U)$ which are those vector fields whose Jacobian at $0$ has distinct eigenvalues.

\begin{Thm}\label{thm:approx}
For any choice of $\Cont^\infty$ smooth vector field $F_0 \in \Xfx(U)$ which converges to $0 \in \R^n$, any choice of finite observation time $T>0$, and any desired accuracy $\delta>0$, there exist:
\begin{enumerate}\packList
    \item a linear system with a diagonal system matrix $\Lambda \in \mathbf{M}_{n\times n}(\C)$ 
    \item a multinomial observable of that system $P: \C^n \to U$
    \item an invertible mapping of initial conditions $S: U \to \C^n$
\end{enumerate}
such that the trajectory of the flow $\kF_0^t$ starting at $\xF$ is $\delta$-close to the $P$ observations of the trajectory of the linear system $\dot \yF = \Lambda \yF$ starting at $S(\xF)$.
\end{Thm}
\begin{proof}

Using \cite[Theorem 5]{kvalheim2020generic}, we learn that $\Gfx(U)$ is (among other things) dense in $\Xfx(U)$ with respect to the infinity norm, i.e. for every $\varepsilon>0$ and $F_0 \in \Xfx(U)$ there is a $F_\varepsilon \in \Gfx(U)$ such that $\|F_\varepsilon-F_0\|_\infty < \varepsilon$.
From \cite[Theorem 16]{kvalheim2020generic}, we obtain that the $n$ principal Koopman eigenfunctions of $F_\varepsilon$ exist, they are uniquely defined, and together they provide a $\Cont^\infty$ diffeomorphism from $\Psi: U \to \C^n$, that conjugates the flow $\kF_\varepsilon$: $\kF^t_\varepsilon = \Psi^{-1} \circ \exp( t \Deriv F_\varepsilon(0)) \cdot \Psi$ for all $t$.
From the Stone-Weierstrass theorem it follows that there exists a multinomial approximation $\Psinv$ to $\Psi^{-1}$ such that within $\Psi(U)$, $\|\Psinv-\Psi^{-1}\|_\infty < \varepsilon$.
This gives $\| \kF^t_\varepsilon - \Psinv \circ \exp( t \Deriv F_\varepsilon(0)) \cdot \Psi \|_\infty < \varepsilon$.
Note that because it is $\Cont^\infty$ on a compact set $U$, $F_0$ is Lipschitz; assume its Lipschitz constant is $L$.
Consider the difference between observing trajectories of the true flow $\kF_0^t$ and trajectories of the multinomial observations $\Psinv$ of the diagonal linear system $\exp( t \Deriv F_\varepsilon(0))$ over a finite time horizon $T$.
There is a well known result bounding the difference between flows using the infinity norm of the difference between the vector fields:
\begin{align*}
    \|\kF^t_0-\kF^t_\varepsilon\|_\infty \leq  \frac{\exp(tL)-1}{L} \|F_0-F_\varepsilon\|_\infty < \varepsilon C
\end{align*}
with $C := (\exp(TL)-1)/L$.
Using the triangle inequality this yields
\begin{align}
    \|\kF^t_0 - \Psinv \circ \exp( t \Deriv F_\varepsilon(0)) \cdot \Psi \|_\infty <  (C+1) \varepsilon
\end{align}
We complete the proof by choosing $\varepsilon < \delta / (C+1)$; diagonalizing $\Deriv F_\varepsilon(0) = V^{-1}\Lambda V$ with its  eigenvector matrix $V$ (which is possible because $F_\varepsilon\in \Gfx(U)$); defining $P := \Psinv \circ V^{-1}$; and defining $S := V \cdot \Psi$.
\end{proof}

\subsection{Diagonal linear systems}\label{sec:diagsys}

We now further strengthen the result of Theorem \ref{thm:approx}, to show what is in some ways a proof that finite Koopman linearizations can approximate all $\Xfx$ systems.

\begin{Thm}\label{thm:flatten}
Consider the linear dynamical system $\dot \xF = \Lambda \xF$ with diagonal system matrix $\Lambda$ on $\sX = \C^n$.
Let $h: \sX \to \C$ be an observable which is a multinomial (multivariate polynomial) of order $m$ in the components of $\xF \in \sX$.
Then there exists an $N\in \N$, $N\geq n$ and a linear system with a diagonal system matrix $\XMat \in \mathbf{M}_{N\times N}(\C)$, a linear observable of that system $H: \C^N \to \C$, and a mapping of initial conditions $\Psi: U \to \C^N$, such that for all $t$ and $\xF$:
 $$h(\exp(t \Lambda)\cdot \xF) = H \cdot \exp(t \XMat) \cdot \Psi(\xF).$$
\end{Thm}

\begin{proof}
The proof is primarily a matter of notation.

Let $\text{diag}(\Lambda) = \lambda_1,\ldots,\lambda_n$:
\begin{equation}
  \exp(t\Lambda)\cdot \xF = \exp(t\Lambda)\cdot (x_1,\ldots,x_n) = (\e^{t \lambda_1}x_1, \ldots, \e^{t \lambda_n}x_n).
\end{equation}

An $h$ can be written as
\begin{equation}
  h(\xF)
     = \sum_{\alpha_1+\ldots+\alpha_n \leq m} H_{(\alpha_1,\ldots,\alpha_n)} x_1^{\alpha_1}\cdots x_n^{\alpha_n} = \sum_{\|\alpha\|_1\leq m} H_\alpha \xF^\alpha
\end{equation}
where the latter uses multi-index notation for $\alpha$.
The number of non-zero $H_\alpha$ coefficients is the $N$ needed.

However, this same class of systems can be viewed in a different light.
Define $\sX^{[m]}$ the space given by monomial features of order $m$ or less from $\sX$, i.e. the image of $\sX$ under the map
\begin{equation}
  \Psi(\xF) :=  \bigoplus_{\|\alpha\|_1\leq m} \xF^\alpha
\end{equation}
where $\|\cdot\|_1$ is the 1-norm.
We denote by $\Psi_\alpha(\xF)$ the $\alpha$ component of this result.
This definition allows us to view $\sum H_\alpha \xF^\alpha$ as an inner product, and $H$ as a linear functional $L(\sX^{[m]},\C)$, giving
\begin{equation}
    h(\xF) = \sum_{\|\alpha\|_1\leq m} H_\alpha \xF^\alpha = H \cdot \Psi(\xF).
\end{equation}
We have ``lifted'' $\sX$ into a feature space $\sX^{[m]}$ in which $h$ admits a linear representation.

Given that $\exp(t\Lambda) \cdot \xF = (\e^{t \lambda_1}x_1, \ldots, \e^{t \lambda_n}x_n)$, note that for any multi-index $\alpha$ we can define $\Xeig{\alpha} :=  \sum_{i=1}^n \lambda_i \alpha_i$ and have
\begin{equation}
  \Psi_\alpha(\exp(t\Lambda) \cdot \xF ) = (\exp(t\Lambda) \cdot \xF )^\alpha
  = \e^{t \lambda_1\alpha_1} x_1^{\alpha_1}\cdots \e^{t \lambda_n\alpha_n} x_n^{\alpha_n} = \e^{t \Xeig{\alpha}} x^\alpha = \e^{t \Xeig{\alpha}} \Psi_\alpha(\xF)
\end{equation}
In other words, if we define a diagonal matrix $\XMat$ whose diagonal elements are $\Xeig{\alpha}$ for the same collection of multi-indices $\alpha$ used to create $\Psi$, $\Psi(\exp(t\Lambda) \cdot \xF) = \exp(t \XMat) \Psi(\xF)$.
\end{proof}

Using Theorem \ref{thm:flatten}, we have shown that the approximation obtained from Theorem \ref{thm:approx} in terms of non-linear observations of a linear system can in fact be viewed as a linear observation of an extended linear system.
This extended linear system has a very special structure -- its eigenvalues fall on the lattice $\Xeig{\alpha}$ generated from the original system's eigenvalues $\lambda$.
It is also initialized only on the (dynamically invariant) algebraic variety which is the image of $\Psi$, $\sX^{[m]} \subset \C^N$.
Nonetheless, the result tells us that all systems in $\Xfx(U)$ can be approximated to any chosen accuracy $\delta>0$ and for any time horizon $T>0$ by a linear observation of a linear system.

\subsection{All 3x3 cases are covered}\label{sec:proof}

\begin{lemma}
Any $3\times3$ real matrix $\mathbf{M}$ with one complex-conjugate pair $\alpha\pm i\beta$ and one real eigenvalue $\lambda$ can be parametrized by \refEqn{A} and \refEqn{QSL} up to an orthogonal similarity transform $\mathbf{U}$. i.e.
\begin{equation}
    \mathbf{M} = \mathbf{UAU}^\T = \mathbf{UQS\Lambda S}^{-1}\mathbf{Q}^{-1}\mathbf{U}^\T.
\end{equation}
\end{lemma}

\begin{proof} 
Any $3\times3$ real matrix $\mathbf{M}$ with this spectrum has three distinct eigenvalues, and therefore, is block diagonalizable.
Let $\mathbf{v}_1$ be a (complex) eigenvector associated with $\alpha+i\beta$ and $\mathbf{v}_2$ be a (real) eigenvector associated with $\lambda$.
The eigendecomposition of $M$ gives $\mathbf{M} = \mathbf{P\Lambda P^{-1}}$, where 
$\mathbf{P} = [\text{Re}(\mathbf{v}_1),\text{Im}(\mathbf{v}_1), \mathbf{v}_2 ]$. 
W.l.o.g., we can transform $\mathbf{M}$ by a rotation matrix $\mathbf{U}_1$, so that the transformed matrix $\mathbf{M'} = \mathbf{U_1MU_1^T} = \mathbf{P'\Lambda P'^{-1}}$ has complex eigenvector pair in $xy$-plane, i.e.

$$\mathbf{P}' = \left[\begin{array}{ccc}
    \text{Re}(\mathrm{v}_{1,1}') & \text{Im}(\mathrm{v}_{1,1}') & \vdots \\
    \text{Re}(\mathrm{v}_{1,2}') & \text{Im}(\mathrm{v}_{1,2}') & \mathbf{v}_2'\\
    0 & 0 & \vdots
\end{array}\right].$$

The proof is the same as showing there always exists an orthogonal similarity transform $\mathbf{U}_2$, \\
s.t. $\mathbf{M' = U}_2 \mathbf{Q S \Lambda} \mathbf{S}^{-1}\mathbf{Q}^{-1}\mathbf{U}^{-1}_2$. 
Let $\mathbf{U}_2$ be a rotation matrix in $xy$-plane:
$$
\mathbf{U}_2 = \left[\begin{array}{ccc}
    \cos(\omega) & -\sin(\omega) & 0\\
    \sin(\omega) & \cos(\omega) & 0\\
    0 & 0& 1
\end{array}\right], \mathbf{U}_2\mathbf{QS} = \left[\begin{array}{ccc}
    s\cos(\omega) & -\sin(\omega) & \sin(\theta)\cos(\omega+\phi)\\
    s\sin(\omega) & \cos(\omega) & \sin(\theta)\sin(\omega+\phi)\\
    0 & 0& \cos(\theta)
\end{array}\right].
$$
The third column of $\mathbf{U}_2\mathbf{QS}$ can match any eigenvector $\mathbf{v}_2'$ by picking angle $\theta$ and $\phi$.
The following Mathematica code demonstrates that we can always pick $\omega$ and $s$ to match the upper-left $2\times2$ block.
In this code, $k_1 = \frac{ly+xz}{lx-yz}, k_2 = \frac{l^2+z^2}{lx-yz}$ and $\left[\begin{array}{cc}
    x & y \\
    z & l
\end{array}\right] =    \left[\begin{array}{cc}\text{Re}(\mathrm{v}_{1,1}') & \text{Im}(\mathrm{v}_{1,1}') \\
\text{Re}(\mathrm{v}_{1,2}') & \text{Im}(\mathrm{v}_{1,2}')\end{array}\right]$. 
Thus a solution satisfying $0<s\leq1$ and $0\leq\omega<2\pi$ always exists.

\begin{tabular}{c}
\hline
\begin{lstlisting}[language=Mathematica]
Lam = {{a,b},{-b,a}};
U ={{Cos[w],-Sin[w]},{Sin[w],Cos[w]}};
T = {{1,1}, {-I,I}};
S ={{s,0},{0,1}};
A = Simplify[U.S.Lam.Inverse[S].Inverse[U]];
Vi = Transpose[Simplify[Eigenvectors[A]]];
Vip = Simplify[Vi.Inverse[T]];
Simplify[Solve[Vip[[1,1]]==k1 && Vip[[1,2]]==k2, {w,s}]]
\end{lstlisting}\\
\hline \\
\end{tabular}
\end{proof}

\subsection{9D linear system with non-resonant modes}
We present a 9D linear system with linear observables that having spectrum similar to the right 9 plots of \refFig{eigDensity1-eigmoving}. 
Each eigenvalue differs less than 0.01. 
The eigenvalue density plots in \refFig{appendix-nonresonant} and \refFig{eigDensity1-eigmoving} are similar as well. 
This shows the eigenvalue estimation error is not cause by eigenvalue resonance with second order observables, but the number and spectrum distance. 

\begin{figure}[!t]\label{fig:appendix-nonresonant}
  \centering
  \includegraphics[width=0.6\textwidth]{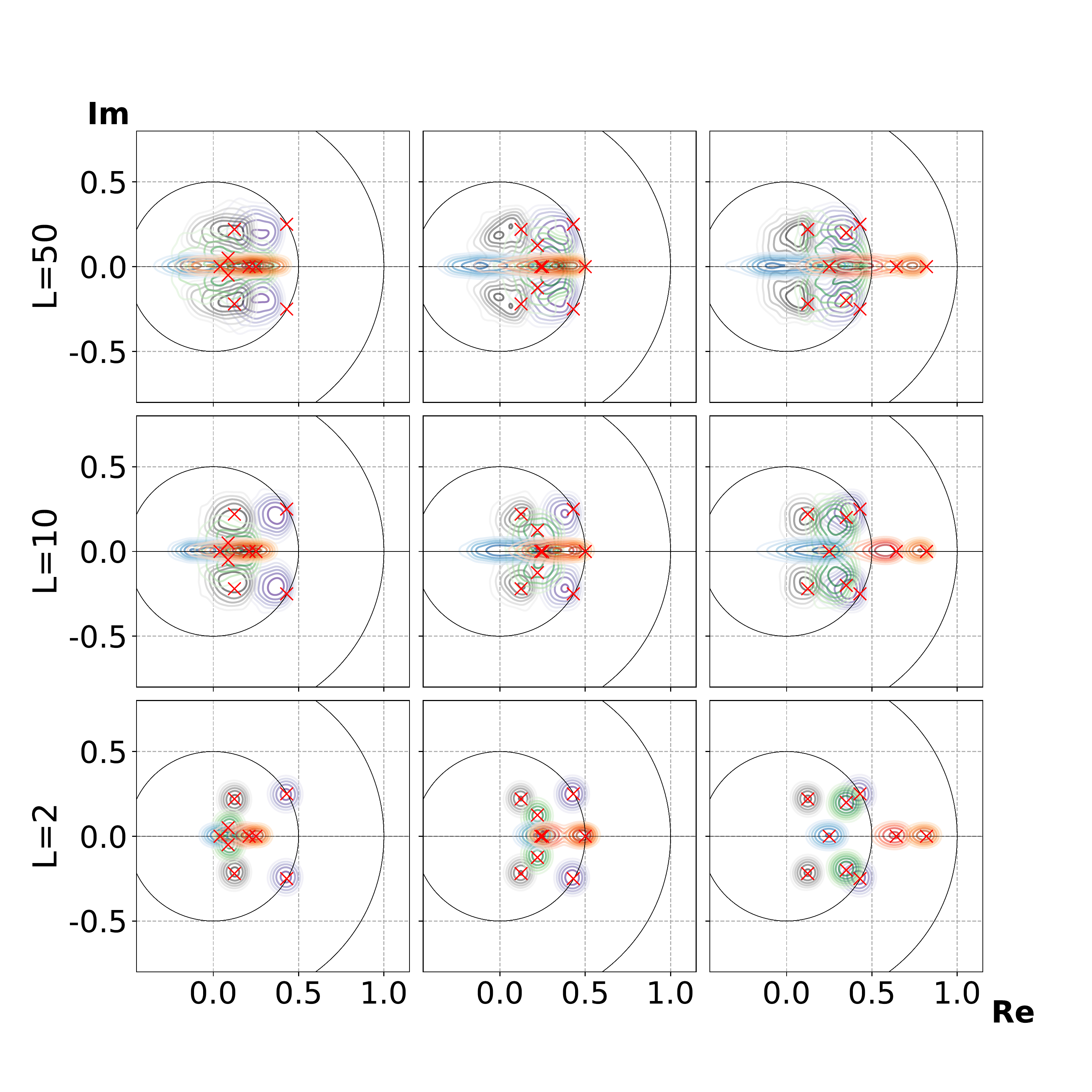}
  \vspace{-0.3in}
  \caption{Eigenvalue density plot using slightly perturbed non-resonant eigenvalues.\
  Different rows have different trajectory length, and different columns have similar eigenvalue distribution as right 9 plots of \refFig{eigDensity1-eigmoving}.\
  Red crosses are ground truth eigenvalues.}
\end{figure}

\subsection{Other DMD algorithms}\label{sec:appendix-algorithms}
We present the results with the same setting as in section \ref{sec:results}, using forward-backward (FB) DMD~\cite{Dawson2016ef}, total-least square (TLS) DMD~\cite{Hemati2017tcfd} and optimized DMD~\cite{Askham2018siads}. 
The difficulties in spectrum recovery, and sensitivity to non-normality and higher order observables persist.
The optimized DMD does not support multiple trajectories as input data, so it only uses one trajectory with specified length throughout the plots. 

In \refFig{app-1}-\refFig{app-4}, we show the contour plots with the same system as the first and fourth column of \refFig{eigDensity1-cond} and \refFig{eigDensity2-cond}.
FB DMD and TLS DMD can correct the underestimation bias caused by observation noise in exact DMD~\cite{Kutz2016book}, but they overestimate the modes under both system and observation noise in \refFig{app-1}. 
We also did not see bias in exact DMD on a dataset with more initial conditions.  
In \refFig{app-3}, FB DMD removes the multi-modal error, and gives better estimation on the real eigenvalue with $L=10,50$. 
In \refFig{app-2} and \refFig{app-4}, all the algorithms encountered difficulties to recovery the full spectrum with second order observables.
Optimized DMD with $L=50$ in \refFig{app-2} gives correct frequency estimation for the slowest decaying oscillatory mode (purple contour), and in \refFig{app-4} the purple contour is surrounding all of the oscillatory modes.
This might indicate the slowest decaying mode from Optimized DMD can always give some reliable information about the dynamics. 
Furthermore, Optimized DMD only used one trajectory with length 50, i.e. half of the data comparing with the other algorithms. 
A promising future work is to make Optimized DMD support multiple trajectories. 

\newpage
\begin{figure}[!t]
  \centering
  \includegraphics[width=0.7\textwidth]{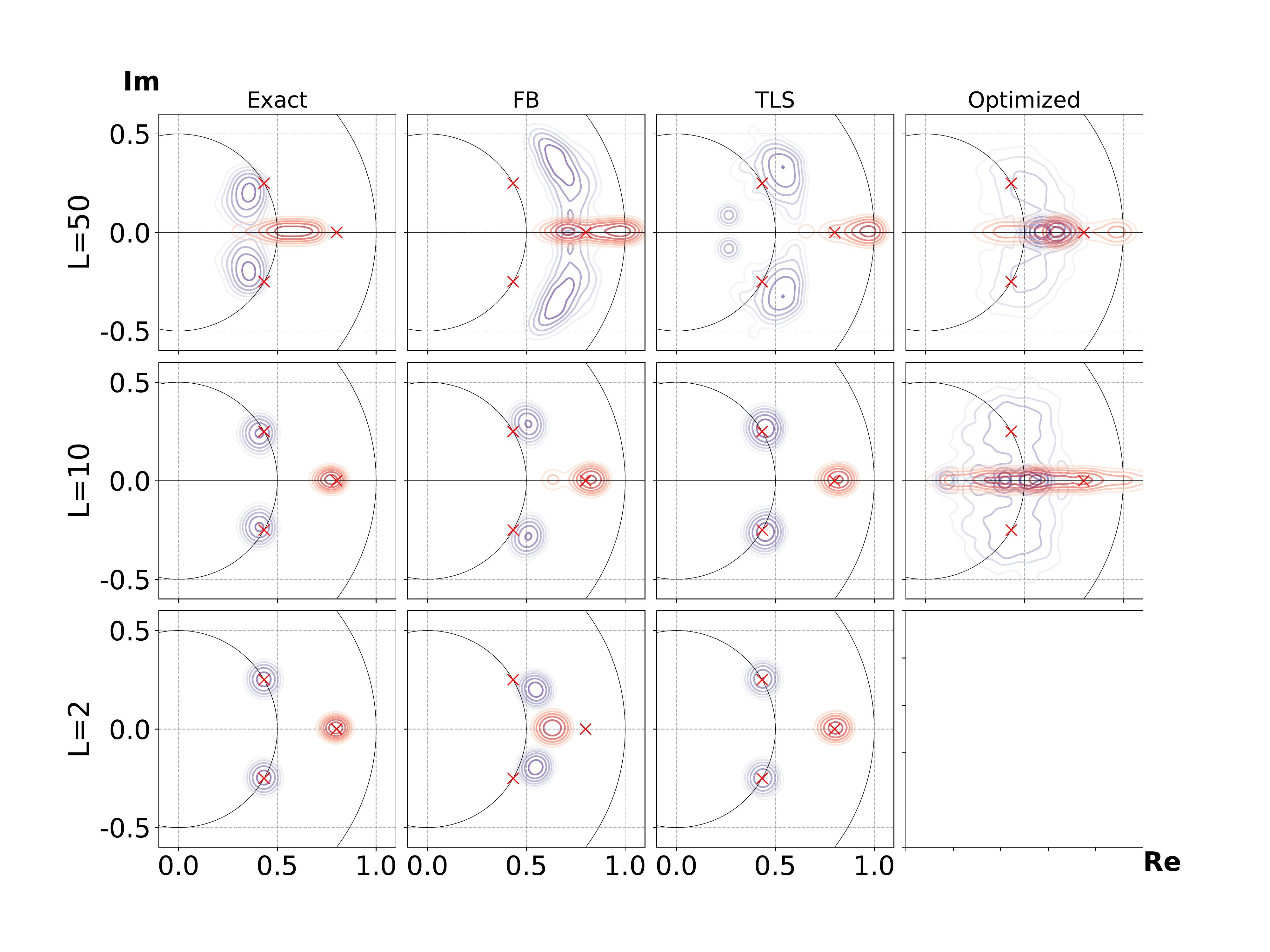}\label{fig:app-1}
  \vspace{-0.3in}
  \caption{Eigenvalue density plot with first order observables.\
  System $A$ matrix is constructed by $\theta=0, \phi=0, s=1$.\
  Different rows have different trajectory length, and different columns use different algorithms.\
  Red crosses are ground truth eigenvalues. Optimized DMD uses only 1 trajectory.}
\end{figure}

\begin{figure}[!bht]
  \centering
  \includegraphics[width=0.7\textwidth]{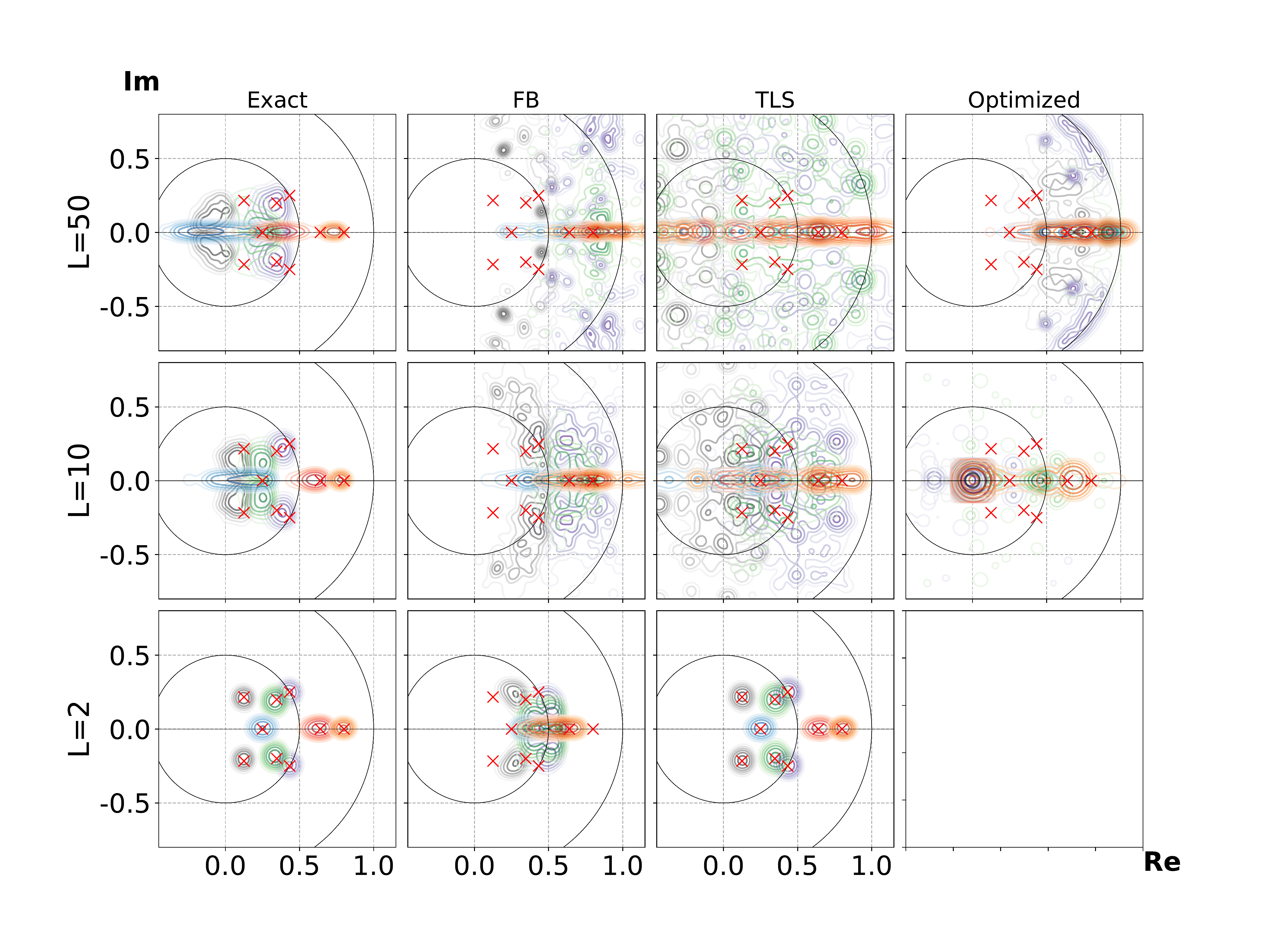}\label{fig:app-2}
  \vspace{-0.3in}
  \caption{Eigenvalue density plot with monomial observables up to 2nd order.\
  System $A$ matrix is constructed by $\theta=0, \phi=0, s=1$.\
  Different rows have different trajectory length, and different columns use different algorithms.\
  Red crosses are ground truth eigenvalues. Optimized DMD uses only 1 trajectory.}
\end{figure}

\begin{figure}[!t]
  \centering
  \includegraphics[width=0.7\textwidth]{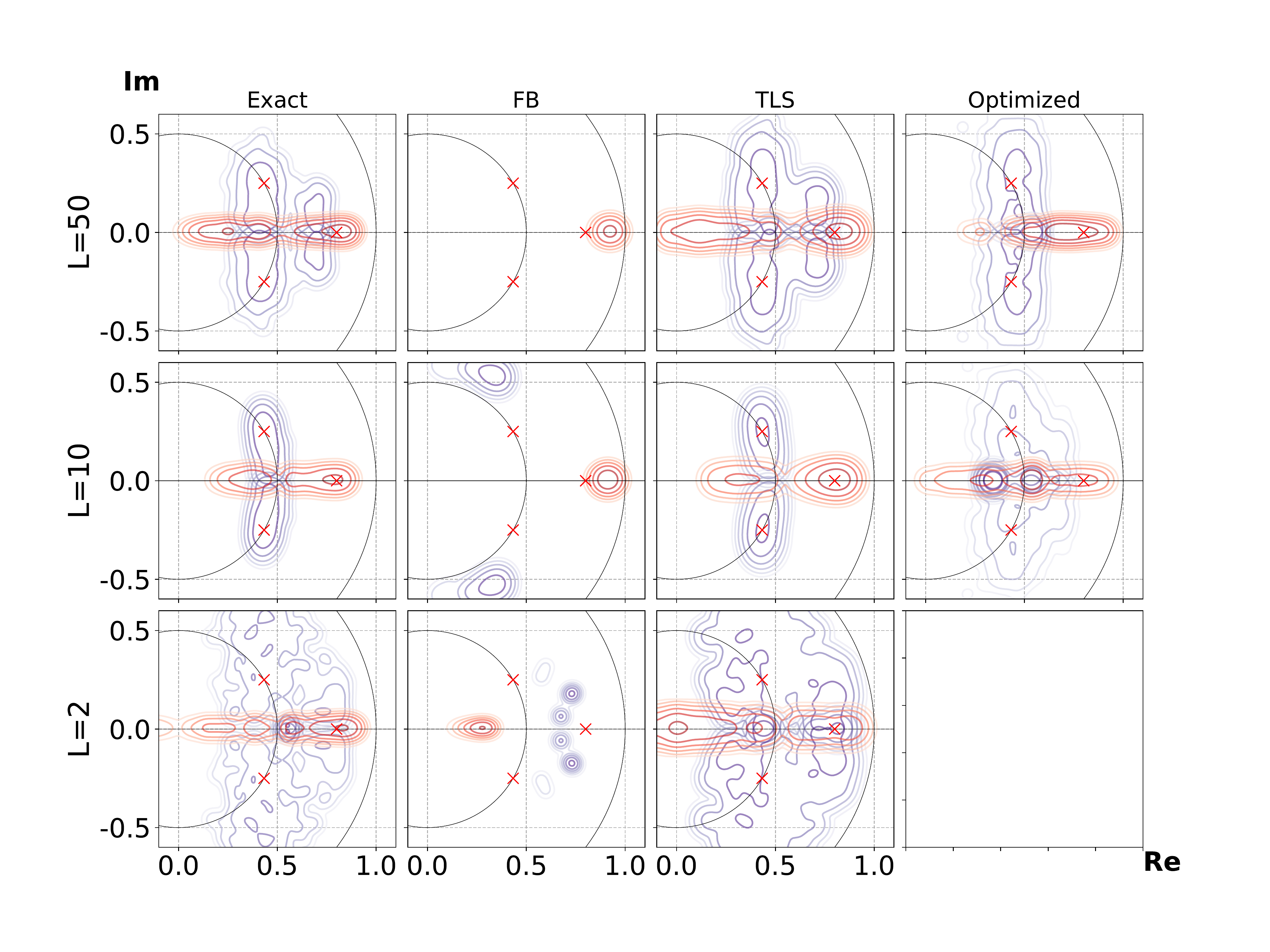}\label{fig:app-3}
  \vspace{-0.3in}
  \caption{Eigenvalue density plot with first order observables.\  
  System $A$ matrix is constructed by $\theta=1.4, \phi=0, s=0.1$.\
  Different rows have different trajectory length, and different columns use different algorithms.\
  Red crosses are ground truth eigenvalues. Optimized DMD uses only 1 trajectory.}
\end{figure}

\begin{figure}[!bht]
  \centering
  \includegraphics[width=0.7\textwidth]{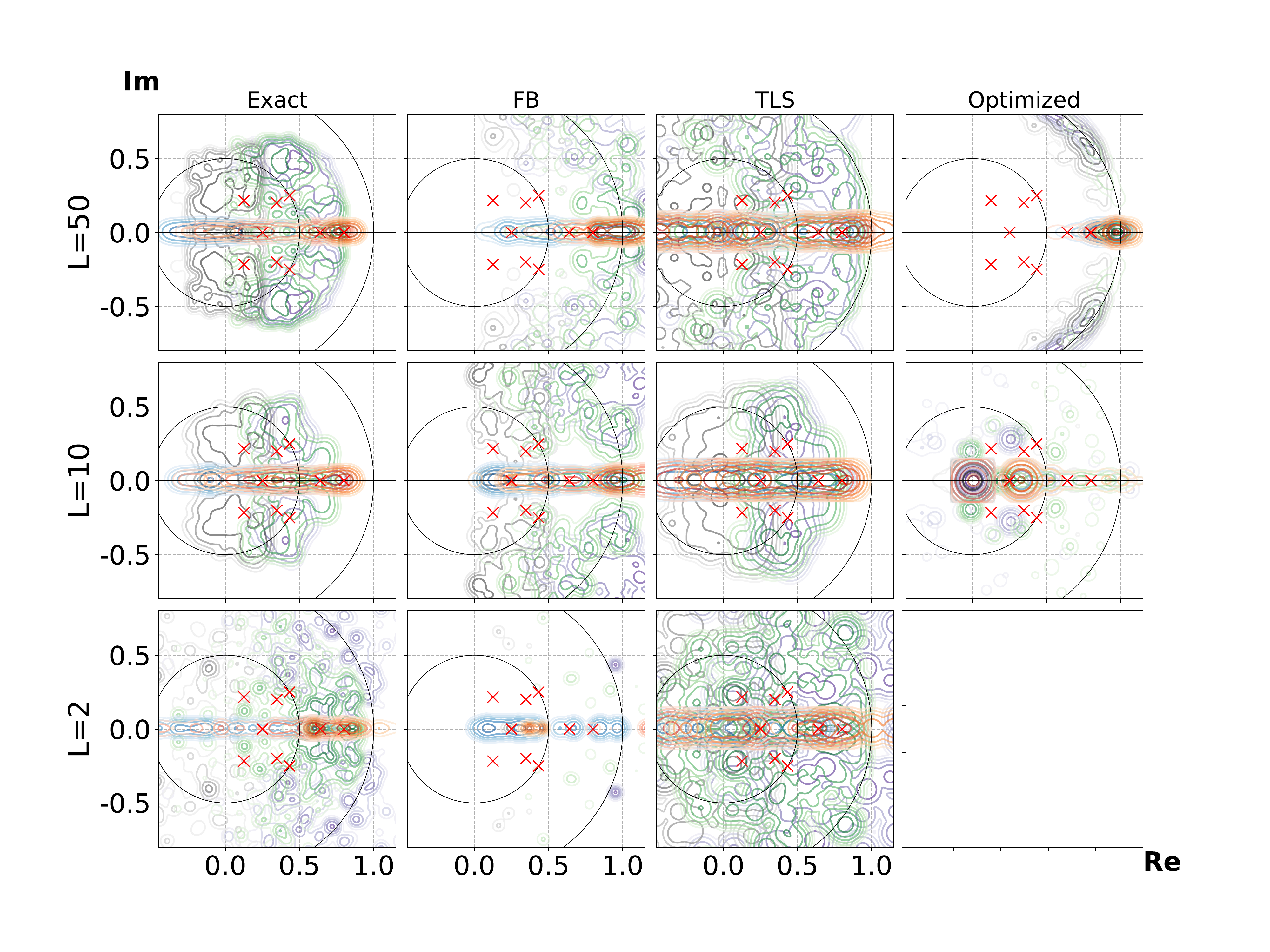}\label{fig:app-4}
  \vspace{-0.3in}
  \caption{Eigenvalue density plot with monomial observables up to 2nd order.\
  System $A$ matrix is constructed by $\theta=1.4, \phi=0, s=0.1$.\
  Different rows have different trajectory length, and different columns use different algorithms.\
  Red crosses are ground truth eigenvalues. Optimized DMD uses only 1 trajectory.}
\end{figure}

\newpage
\subsection{5D system}
We present a 5D system with spectrum more reminiscent of multi-legged systems, shown in \refFig{app-5}. 

\begin{figure}[!b]
  \centering
  \includegraphics[width=0.7\textwidth]{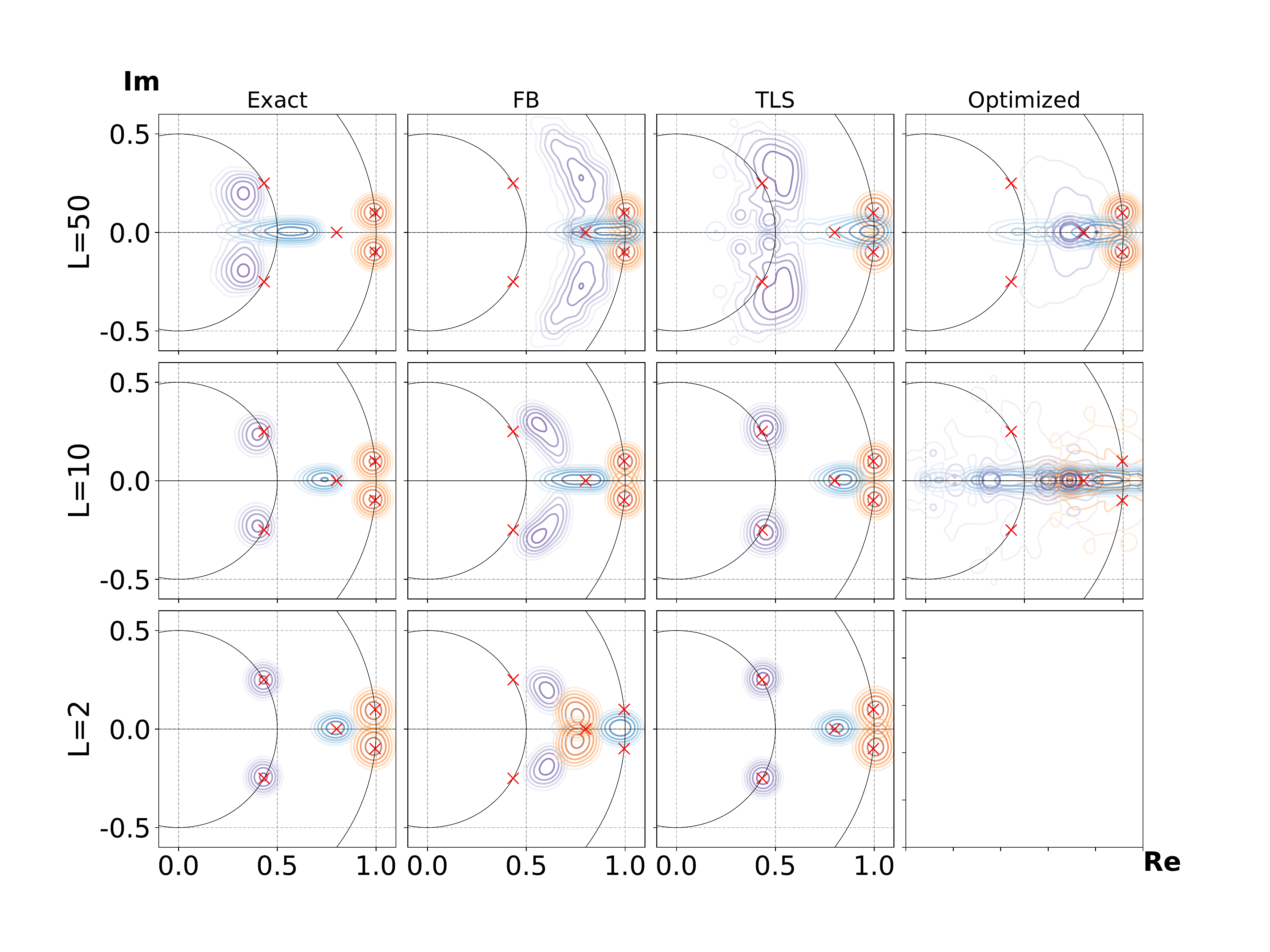}
  \vspace{-0.3in}
  \caption{Eigenvalue density plot with two slow oscillatory modes added (5d system with linear observation). \
  Different rows have different trajectory length, and different columns use different algorithms.\
  Red crosses are ground truth eigenvalues.}\label{fig:app-5}
\end{figure}

\end{document}